\documentclass[a4paper]{amsart}

\usepackage[british]{babel}

\usepackage{amsmath,amssymb,amsthm}
\usepackage{graphicx}

\providecommand{\dd}{\mathrm{d}}
\DeclareMathOperator{\dom}{dom}
\DeclareMathOperator{\lspan}{span}

\theoremstyle{theorem}
\newtheorem{theorem}{Theorem}

\newtheorem{corollary}[theorem]{Corollary}
\newtheorem{proposition}[theorem]{Proposition}
\theoremstyle{definition}
\newtheorem{definition}[theorem]{Definition}
\newtheorem{assumption}[theorem]{Assumption}
\theoremstyle{remark}
\newtheorem{remark}[theorem]{Remark}
\newtheorem{example}[theorem]{Example}

\numberwithin{equation}{section}
\numberwithin{theorem}{section}

\begin{document}
\title[Cubature Methods for SPDEs]{Cubature Methods For Stochastic (Partial) Differential Equations In Weighted Spaces}
\author{Philipp D\"orsek \and Josef Teichmann \and Dejan Velu\v{s}\v{c}ek}
\address{ETH Zurich, D-MATH, R\"amistrasse 101, CH-8092 Zurich, Switzerland}
\email{\{philipp.doersek, josef.teichmann, dejan.veluscek\}@math.ethz.ch}
\label{firstpage}
\begin{abstract}
The cubature on Wiener space method, a high-order weak approximation scheme, is established for SPDEs in the case of unbounded characteristics and unbounded payoffs. We first introduce a recently described flexible functional analytic framework, so called weighted spaces, where Feller-like properties hold. A refined analysis of vector fields on weighted spaces then yields optimal convergence rates of cubature methods for stochastic partial differential equations of Da Prato-Zabczyk type. The ubiquitous stability for the local approximation operator within the functional analytic setting is proved for SPDEs, however, in the infinite dimensional case we need a newly introduced assumption on weak symmetry of the cubature formula. In finite dimensions, we use the UFG condition to obtain optimal rates of convergence on non-uniform meshes for nonsmooth payoffs with exponential growth.
\end{abstract}
\maketitle
\section{Introduction}
Cubature on Wiener space, a realization of the abstract KLV high order method after Shigeo Kusuoka \cite{Kusuoka2001}, Terry Lyons and Nicolas Victoir \cite{LyonsVictoir2004}, is a weak approximation scheme for stochastic differential equations. Significant advantages in comparison to other weak approximation schemes such as Taylor methods, see \cite{KloedenPlaten1992}, are that it respects the geometry of the problem, and that at least theoretically, it is possible to reach arbitrarily high rates of convergence without requiring the calculation of higher derivatives, see \cite[Theorem~2.4, Proposition~2.5]{LyonsVictoir2004}. The concrete construction of such cubature paths of high order is still quite difficult, see \cite{GyurkoLyons2011} for paths up to order $11$ for a single driving Brownian motion. Cubature schemes provide a time-discretization approximating the unknown expected value of a functional of the solution process of the SPDE by an expectation of an iteratively constructed function on a high-dimensional discrete product space. Often a direct evaluation of the functional on the discrete probability space is too expensive, therefore, several methods to speed-up the evaluation of cubature schemes such as recombination \cite{LittererLyons2007,SchmeiserSoreffTeichmann2007} or tree-based branching \cite{CrisanLyons2002} have emerged. Otherwise the functionals have to be evaluated with Monte Carlo or Quasi Monte Carlo algorithms on the discrete product probability space.

High-order weak approximation schemes provide interesting lower complexity alternatives to standard multi-level Monte Carlo schemes if Quasi Monte Carlo algorithms or deterministic algorithms can be applied for the evaluation of the constructed functionals. Indeed, multi-level Monte Carlo schemes lead to complexity estimates of order (almost) $ \mathcal{O}(\epsilon^{-2}) $, i.e.~to reach accuracy $ \epsilon $ a number of operations of order $ \epsilon^{-2} $ is necessary. In contrast, QMC evaluations of weak, high-order approximation schemes of order $ k $ lead to complexity estimates of order (almost) $ \mathcal{O}(\epsilon^{-1-1/k}) $, as long as the QMC integration yields optimal convergence (this in turn also depends on the dimension of integration space, which is moderate for high order methods). Hence we believe that it is worth analyzing in depth the functional analytic framework of cubature schemes, i.e.~we aim for constructing a flexible enough pool of Banach spaces of payoffs and Banach spaces of characteristics, where relevant problems from practice can be embedded.

In this work, we shall relax the regularity assumptions of the cubature method, similarly as was done in \cite{DoersekTeichmann2010,Doersek2011phdthesis,DoersekTeichmann2011} for the splitting approach of Syoiti Ninomiya and Nicolas Victoir \cite{NinomiyaVictoir2008}.
Consider a stochastic differential equation on $\mathbb{R}^n$ in its Stratonovich form,
\begin{equation}
  \dd X^x_t = \sum_{i=0}^{d}V_j(X^x_t)\circ\dd B^{j}_t.
\end{equation}
All initial work was based on the fundamental assumption that the vector fields $V_j\colon\mathbb{R}^n\to\mathbb{R}^n$ are bounded and $\mathrm{C}^{\infty}$-bounded. This is also a typical assumption in other approximation methods for stochastic differential equations, e.g., in \cite{TalayTubaro1991}. Some success in relaxing these assumptions, which are actually rarely satisfied in practical problems was achieved, at least for approximations of the splitting type, in works by Tanaka and Kohatsu-Higa \cite{TanakaKohatsuHiga2009} and Alfonsi \cite{Alfonsi2010}. While in the first one the focus was on extensions to L\'evy driving noise and in the second to CIR processes, in both approaches it was recognized that polynomially bounded payoffs are the correct context for problems with Lipschitz continuous vector fields.

Another approach was suggested in \cite{DoersekTeichmann2010}.
There, splitting schemes were analyzed on general weighted spaces, allowing in particular the approximation of Da Prato-Zabczyk stochastic partial differential equations where the drift part, the infinitesimal generator of a strongly continuous semigroup on the infinite dimensional state space, is not even continuous.

All these approaches profited from the special structure of splitting schemes, as there the stability or power boundedness of the discrete approximation operator can be shown by investigating every part separately.
Instead, we follow a similar idea as was applied to the stochastic Navier-Stokes equations in \cite{Doersek2011}.
We extend the results of \cite{BayerTeichmann2008} to more general coefficients and payoffs.
This allows us to obtain methods of order higher than 2 without having to resort to extrapolation, see \cite{BlanesCasas2005,OshimaTeichmannVeluscek2009}.
While the use of the weighted spaces from \cite{RoecknerSobol2006,DoersekTeichmann2010} is also mandatory here, we shall provide a refined analysis of the vector fields defined on these spaces.
This will allow us to do a Taylor expansion of the cubature approximations to compute the local approximation order.

While dealing with the stability, we shall use two different approaches.
In the finite dimensional case with sufficiently smooth vector fields, the Gronwall inequality yields the claim in a straightforward manner under a reasonable assumption of compatibility between the vector fields and the weight function.
In the infinite dimensional case, we apply the method of the moving frame from \cite{Teichmann2009}.
This leads to time dependent vector fields that are nonsmooth in the time component.
As this makes a Taylor expansion impossible, we introduce a \emph{weak symmetry condition} on cubature paths, an assumption usually satisfied by cubature schemes.
This allows us to obtain stability not only for Da Prato-Zabczyk equations with pseudocontractive generator, but also for stochastic differential equations on infinite dimensional state spaces, where the vector fields depend roughly, i.e., continuously, but not differentiably, on time.

Finally, we consider the effects of the UFG condition in our setting.
Under the same assumptions on the coefficients as in \cite{CrisanGhazali2007}, we are able to prove optimal rates of convergence on non-uniform meshes for nonsmooth payoffs that are allowed to grow exponentially.

There are many successful discretisation schemes for stochastic partial differential equations.
\cite{JentzenKloeden2009} gives an overview of strong and pathwise schemes.
Weak approximation schemes are more difficult.
Recently, it was proved in \cite{Debussche2011} that an implicit Euler scheme converges almost with weak rate $1/2$ for equations driven by space-time white noise, doubling the corresponding strong rate of convergence; see also the references in \cite{Debussche2011} for more background on weak approximation schemes for stochastic partial differential equations with space-time white noise.
In contrast, we restrict ourselves to finite-dimensional driving noise, but obtain the same weak rate of convergence as for finite-dimensional state spaces.

In our proofs, $C$ denotes a generic positive real constant that can change from line to line.

\section{$\mathcal{B}^{\psi}$ spaces}
We recall the following definition of spaces of functions with controlled growth, see also \cite{RoecknerSobol2006,DoersekTeichmann2010,Doersek2011,Doersek2011phdthesis,DoersekTeichmann2011}. Notice that we obtain Feller-like properties for SPDEs in this setting.
\begin{definition}
  Let $(X,\lVert\cdot\rVert_{X})$ be the dual space of a separable Banach space, and $\varphi\colon X\to (0,\infty)$ be bounded from below by some $\delta>0$.
  For a Banach space $(Y,\lVert\cdot\rVert_{Y})$, we set
  \begin{equation}
    \mathrm{B}^{\varphi}(X;Y)
    :=
    \left\{ f\colon X\to Y\colon\sup_{x\in X}\varphi(x)^{-1}\lVert f(x)\rVert_{Y}<\infty \right\},
  \end{equation}
  endowed with the $\varphi$-norm
  \begin{equation}
    \lVert f\rVert_{\varphi}
    :=
    \sup_{x\in X}\varphi(x)^{-1}\lVert f(x)\rVert_{Y}.
  \end{equation}
  Let $k\ge 0$.
  If $\varphi=(\varphi_j)_{j=0,\dots,k}$, $\varphi_j\colon X\to(0,\infty)$ bounded from below by some $\delta>0$, $j=0,\dots,k$, we set
  \begin{alignat}{2}{}
    \mathrm{B}^{\varphi}_k(X;Y)
    :=
    \bigl\{ f\in\mathrm{C}^{k}(X;Y)\colon 
    &\text{$\sup_{x\in X}\varphi_j(x)^{-1}\lVert D^j f(x)\rVert_{L_j(X;Y)}<\infty$} \notag \\
    &\text{for $j=0,\dots,k$} \bigl\}.
  \end{alignat}
  $\mathrm{B}^{\varphi}_{k}(X;Y)$ is endowed with the norm
  \begin{equation}
    \lVert f\rVert_{\varphi,k}
    :=
    \lVert f\rVert_{\varphi_0} + \sum_{j=1}^{k}\lvert f\rvert_{\varphi_j,j},
  \end{equation}
  where the seminorms $\lvert\cdot\rvert_{\varphi_j,j}$ are given by
  \begin{equation}
    \lvert f\rvert_{\varphi_j,j}
    :=
    \sup_{x\in X}\varphi_j(x)^{-1}\lVert D^j f(x)\rVert_{L_j(X;Y)}.
  \end{equation}
  Here, $L_j(X;Y)$ denotes the space of bounded multilinear forms $a\colon X^j\to Y$, and is endowed with the norm
  \begin{equation}
    \lVert a\rVert_{L_j(X;Y)}
    :=
    \sup_{\lVert h_i\rVert\le 1, i=1,\dots,j}\lVert a(h_1,\dots,h_j)\rVert_{Y}.
  \end{equation}
	For simplicity, we set $L_0(X;Y):=Y$; we remark that $L_1(X;Y)$ is the space of bounded linear operators $X\to Y$, and in this case, the above norm is the usual operator norm.
  If $Y=\mathbb{R}$, we define $\mathrm{B}^{\varphi}(X):=\mathrm{B}^{\varphi}(X;\mathbb{R})$ and $\mathrm{B}^{\varphi}_k(X):=\mathrm{B}^{\varphi}_k(X;\mathbb{R})$.
\end{definition}
\begin{definition}
  \label{def:admissibleweightfunction}
  Let $(X,\lVert\cdot\rVert_{X})$ be the dual space of a separable Banach space.
  A function $\varphi$ is called \emph{admissible weight function} if and only if $\varphi\colon X\to(0,\infty)$ is such that $K_R:=\left\{ x\in X\colon\varphi(x)\le R \right\}$ is weak-$*$ compact for all $R>0$.

  It is called \emph{D-admissible weight function} if and only if it is an admissible weight function and for every $x\in X$, there exists some $R>0$ such that $B_\varepsilon(x)\subset K_R$ for some $\varepsilon>0$, where $B_\varepsilon(x):=\left\{ y\in X\colon \lVert y-x\rVert_X\le\varepsilon \right\}$ is the closed $\varepsilon$-ball around $x$.

	It is called \emph{C-admissible weight function} if and only if $\varphi$ is bounded from below, weak-$*$ lower semicontinuous, and if for every $x\in X$, there exists some $\varepsilon>0$ such that $\varphi$ is bounded on $B_{\varepsilon}(x)$.
\end{definition}
\begin{remark}
  We do not require C-admissible weight functions to be admissible.
  However, $\varphi$ is D-admissible if and only if it is admissible and C-admissible.
\end{remark}
\begin{theorem}
  \label{thm:Bpsikcomplete}
  Let $k\in\mathbb{N}$, and assume that $\varphi=(\varphi_j)_{j=0,\dots,k}$ is a vector of C-admissible weight functions.
  Then, $\mathrm{B}^{\varphi}_k(X;Y)$ is a Banach space.
\end{theorem}
\begin{proof}
  Let $(f_n)_{n\in\mathbb{N}}$ be a Cauchy sequence in this space.
  It is clear that $f_n$ admits a pointwise limit $f$.
  Moreover, it follows that for every $x\in X$ and every closed $\varepsilon$-ball $B_{\varepsilon}(x)$, $f_n|_{B_{\varepsilon}(x)}$ are Cauchy sequences in $\mathrm{C}^k(B_{\varepsilon}(x);Y)$.
  But this entails that $f|_{B_{\varepsilon}(x)}\in\mathrm{C}^k(B_{\varepsilon}(x);Y)$.
  As differentiability is a local property, we see that $f\in\mathrm{C}^k(X;Y)$.
  The necessary estimates for $f$ and its derivatives are now easy to see.
\end{proof}
\begin{remark}
	A counterexample showing the necessity of C-admissibility in Theorem~\ref{thm:Bpsikcomplete} is given in Appendix~\ref{sec:counterexample}.
\end{remark}
\begin{definition}
  Let $(X,\lVert\cdot\rVert_{X})$ be the dual space of a separable Banach space, its predual being $W$, $X=W^{*}$, and $(Y,\lVert\cdot\rVert_{Y})$ a Banach space.
  The space of \emph{bounded smooth cylindrical functions} is defined by
  \begin{alignat}{2}{}
    \mathcal{A}(X,Y)
    :=
    \bigl\{ 
    	f\colon X\to Y
	\colon 
	&\text{$f=g(\langle \cdot,w_1\rangle,\dots,\langle \cdot,w_n\rangle)$} \notag \\
	&\text{for some $g\in\mathrm{C}_b^\infty(\mathbb{R}^n;Y)$}, \notag \\
	&\text{$w_i\in W$, $i=1,\dots,n$, $n\in\mathbb{N}$} \bigr\}.
  \end{alignat}
	Here, $\langle\cdot,\cdot\rangle$ denotes the dual pairing of $X$ and $W$.
  For $Y=\mathbb{R}$, we set $\mathcal{A}(X):=\mathcal{A}(X,\mathbb{R})$.
\end{definition}
\begin{definition}
  Let $(X,\lVert\cdot\rVert_{X})$ be the dual space of a separable Banach space and $(Y,\lVert\cdot\rVert_{Y})$ be a Banach space. 
  Let $\psi$ be an admissible weight function on $X$.

  The space $\mathcal{B}^{\psi}(X;Y)$ is the closure of $\mathcal{A}(X,Y)$ in $\mathrm{B}^{\psi}(X;Y)$.
  For $Y=\mathbb{R}$, we set $\mathcal{B}^{\psi}(X):=\mathcal{B}^{\psi}(X;\mathbb{R})$.
\end{definition}
\begin{remark}
  \cite[Theorem~4.2]{DoersekTeichmann2010} shows that our definition of $\mathcal{B}^{\psi}(X)$ here agrees with our earlier definition from \cite[Definition~2.2]{DoersekTeichmann2010}.
  Due to \cite[Theorem~2.7]{DoersekTeichmann2010}, the functions in $\mathcal{B}^{\psi}(X)$ are characterized by the property that both $f|_{K_R}\in\mathrm{C}( (K_R)_{w*} )$ and
  \begin{equation}
    \lim_{R\to\infty}\sup_{x\in X\setminus K_R}\psi(x)^{-1}\lvert f(x)\rvert = 0.
  \end{equation}
\end{remark}

\begin{definition}
  Let $(X,\lVert\cdot\rVert_{X})$ be the dual space of a separable Banach space and $(Y,\lVert\cdot\rVert_{Y})$ be a Banach space.
  Let $\psi=(\psi_j)_{j=0,\dots,k}$ with $\psi_j$ D-admissible weight functions for $j=0,\dots,k$.
  The space $\mathcal{B}^{\psi}_{k}(X;Y)$ is the closure of $\mathcal{A}(X,Z)$ in $\mathrm{B}^{\psi}_k(X;Y)$.
  For $Y=\mathbb{R}$, we set $\mathcal{B}^{\psi}_k(X):=\mathcal{B}^{\psi}_k(X;\mathbb{R})$. In particular, by Theorem~\ref{thm:Bpsikcomplete}, it follows that $\mathcal{B}^{\psi}_k(X)$ is a separable Banach space.
\end{definition}

One essential property of $ \mathcal{B}^\psi(X) $ spaces is that the dual space of this separable Banach space is a well understood space of Radon measures, such as in the case of $ C_0(X) $ for locally compact spaces $ X $.

\begin{theorem}[Riesz representation for $\mathcal{B}^\psi(X)$]\label{theorem:rieszrepresentation} 
Let $\ell\colon\mathcal{B}^\psi(X)\to\mathbb{R}$ be a continuous linear functional. Then, there exists a finite signed Radon measure $\mu$ on $X$ such that
\begin{equation}
\ell(f)=\int_{X}f(x)\mu(\dd x)\quad\text{for all $f\in\mathcal{B}^\psi(X)$.}
\end{equation}
Furthermore,
\begin{equation}
\label{eq:rieszrepresentation-psiintbound}
\int_{X}\psi(x)\lvert \mu\rvert(\dd x) = \lVert \ell\rVert_{L(\mathcal{B}^\psi(X),\mathbb{R})},
\end{equation}
where $\lvert\mu\rvert$ denotes the total variation measure of $\mu$.
\end{theorem}
As every such measure defines a continuous linear functional on $\mathcal{B}^\psi(X)$, this completely characterizes the dual space of $\mathcal{B}^\psi(X)$.

This allows for the introduction of the \emph{generalized Feller property}, such that we can speak about strongly continuous semigroups on spaces of functions with growth controlled by $ \psi $, in particular functions which are in general unbounded.

Let $(P_t)_{t\ge 0}$ be a family of bounded linear operators $P_t\colon\mathcal{B}^{\psi}(X)\to\mathcal{B}^{\psi}(X)$ with the following properties:
\begin{enumerate}
		\renewcommand{\theenumi}{{\bf F\arabic{enumi}}}
	\item
		\label{enu:defgenfeller-0id}
		$P_0=I$, the identity on $\mathcal{B}^{\psi}(X)$,
	\item
		$P_{t+s}=P_tP_s$ for all $t$, $s\ge 0$,
	\item
		\label{enu:defgenfeller-pwconv}
		for all $f\in\mathcal{B}^{\psi}(X)$ and $x\in X$, $\lim_{t\to 0+}P_t f(x)=f(x)$,
	\item
		\label{enu:defgenfeller-bound}
		there exist a constant $C\in\mathbb{R}$ and $\varepsilon>0$ such that for all $t\in [0,\varepsilon]$, $\lVert P_t\rVert_{L(\mathcal{B}^{\psi}(X))}\le C$,
	\item
		\label{enu:defgenfeller-positivity}
		$P_t$ is positive for all $t\ge 0$, that is, for $f\in\mathcal{B}^{\psi}(X)$, $f\ge 0$, we have $P_t f\ge 0$.
\end{enumerate}
Alluding to \cite[Chapter~17]{Kallenberg1997}, such a family of operators will be called a \emph{generalized Feller semigroup}.

We shall now prove that semigroups satisfying \ref{enu:defgenfeller-0id} to \ref{enu:defgenfeller-bound} are actually strongly continuous, a direct consequence of Lebesgue's dominated convergence theorem with respect to the measure existing due to Riesz representation.
\begin{theorem}
	\label{theorem:Ttstrongcont}
	Let $(P_t)_{t\ge 0}$ satisfy \ref{enu:defgenfeller-0id} to \ref{enu:defgenfeller-bound}. 
	Then, $(P_t)_{t\ge 0}$ is strongly continuous on $\mathcal{B}^{\psi}(X)$, that is,
	\begin{equation}
		\lim_{t\to 0+}\lVert P_t f-f\rVert_{\psi}=0
		\quad\text{for all $f\in\mathcal{B}^{\psi}(X)$}.
	\end{equation}
\end{theorem}
\begin{proof}
	By \cite[Theorem I.5.8]{EngelNagel2000}, we only have to prove that $t\mapsto \ell(P_t f)$ is right continuous at zero for every $f\in\mathcal{B}^{\psi}(X)$ and every continuous linear functional $\ell\colon\mathcal{B}^{\psi}(X)\to\mathbb{R}$. Due to Theorem~\ref{theorem:rieszrepresentation}, we know that there exists a signed measure $\nu$ on $X$ such that $\ell(g)=\int_{X}g\dd\nu$ for every $g\in\mathcal{B}^{\psi}(X)$. By \ref{enu:defgenfeller-bound}, we see that for every $t\in[0,\varepsilon]$,
	\begin{equation}
		\lvert P_t f(x)\rvert
		\le C \psi(x).
	\end{equation}
Due to \eqref{eq:rieszrepresentation-psiintbound}, the dominated convergence theorem yields
	\begin{alignat}{2}{}
		\lim_{t\to 0+}\int_{X}P_t f(x)\nu(\dd x) = \int_{X}f(x)\nu(\dd x),
	\end{alignat}
	and the claim follows.
	Here, the integrability of $ \psi $ with respect to the total variation measure $ \lvert\nu\rvert $ enters in an essential way.
\end{proof}

\section{Vector fields and directional derivatives}

When we ask for convergence rates we have to specify large enough sets of test functions within the basic $ \mathcal{B}^\psi(X)$-spaces. For this purpose we need to analyze directional derivatives and their functional analytic behavior. This can be done within the setting of $ \mathcal{B}_k^\psi(X;Y) $ spaces.

Let $(X,\lVert\cdot\rVert_{X})$ be the dual space of a separable Banach space. Given $(Z,\lVert\cdot\rVert_{Z})$ the dual space of another separable Banach space that is embedded in $X$, we derive conditions on $V\colon Z\to X$ such that the \emph{directional derivative} $g\in\mathcal{B}^{\hat{\psi}}_{k-1}(Z)$, where
\begin{equation}
	g(z) := Df(z)(V(z))
	\quad\text{for $z\in Z$}
\end{equation}
and $f\in\mathcal{B}^{\psi}_{k}(X)$.
Here, $\psi=(\psi_j)_{j=0,\dots,k}$ and $\hat{\psi}=(\hat{\psi}_j)_{j=0,\ldots,k-1}$ are vectors of D-admissible weight functions on $X$ and $Z$, respectively.

We shall assume that $V\in\mathrm{B}^\varphi_{k-1}(Z;X)$ for some vector $\varphi=(\varphi_j)_{j=0,\ldots,k-1}$ of C-admissible weight functions on $Z$.
Then, $V$ is $k-1$ times continuously Fr\'echet differentiable.
As $f\in\mathrm{C}^{k}(X)$, the Leibniz rule yields
\begin{equation}
	D^j g(z)(h_1,\cdots,h_j)
	=
	\sum_{i=0}^{j}\frac{1}{i!(j-i)!}\sum_{\sigma\in\mathcal{S}_j}g_{j,i}(z,h_{\sigma_1},\cdots,h_{\sigma_j}),
	\quad j=0,\dots,k-1.
\end{equation}
Here, $\mathcal{S}_j$ denotes the symmetric group with $j$ elements, and
\begin{equation}
	\label{eq:vectorfields-defngji}
	g_{j,i}(z,h_1,\cdots,h_j) := D^{i+1}f(z)(h_{1},\cdots,h_{i}, D^{j-i} V(z)(h_{i+1},\cdots,h_{j}) ).
\end{equation}
In particular, if we assume that for some constant $C>0$,
\begin{equation}
	\label{eq:vectorfields-phipsirelation}
	\hat{\psi}_j(z) 
	\ge
	C^{-1}\sum_{i=0}^{j}\binom{j}{i}\psi_{i+1}(z)\varphi_{j-i}(z)
	\quad\text{for $j=0,\cdots,k-1$},
\end{equation}
it follows that $g\in\mathrm{B}^{\hat{\psi}}_{k-1}(Z)$.

It is not so straightforward to prove that $g$ can also be approximated by functions in $\mathcal{A}(X)$, which would imply $g\in\mathcal{B}^{\hat{\psi}}_{k-1}(Z)$.
In \cite{Doersek2011phdthesis}, a general theory for multiplication operators on $\mathcal{B}^\psi$ spaces is derived.
Here, we take a different route, focusing on the problem at hand.
The following definition is essential.
\begin{definition}
	\label{def:vectorfield}
	Given a Banach space $(X,\lVert\cdot\rVert_X)$ and the dual space $(Z,\lVert\cdot\rVert_Z)$ of a separable Banach space.
	Let $V\in\mathrm{B}^{\varphi}_k(Z;X)$ with $\varphi$ a given vector of C-admissible weight functions on $Z$.
	We say that $V\in\mathcal{C}^{\varphi}_k(Z;X)$ if and only if for every $y\in X^*$, there exists a constant $C_{V,y}>0$ such that for all $R>0$, there exists a sequence $(v_n)_{n\in\mathbb{N}}\subset\mathcal{A}(Z)$ with $\sup_{n\in\mathbb{N}}\lVert v_n\rVert_{\varphi,k}\le C_{V,y}$ such that, with $v:=y\circ V$,
	\begin{equation}
		\lim_{n\to\infty}\lVert v - v_n \rVert_{\mathrm{C}^k(B_R(0))}=0.
	\end{equation}
	Here, $B_R(0)$ is the closed unit ball of radius $R$ in $Z$, and
	\begin{equation}
		\lVert g\rVert_{\mathrm{C}^k(B_R(0))}
		:=
		\sum_{j=0}^{k}\sup_{z\in B_R(0)}\lVert D^j g(z)\rVert_{L_j(Z)}.
	\end{equation}
\end{definition}
\begin{remark}
	It is clear that vector fields such as those from \cite[Section~2.2]{DoersekTeichmann2011} satisfy the above assumption.
	More generally, if $Z$ is a Hilbert space and is compactly embedded into a larger Hilbert space $Y$ such that $y\circ V$ can be extended to a smooth mapping $Y\to\mathbb{R}$ lying in $\mathrm{C}_b^k(Y;\mathbb{R})$ for all $y\in X^{*}$, then the above assumption is satisfied, i.e., $V\in\mathcal{C}^{\varphi}_k(Z;X)$ for every vector $\varphi$ of C-admissible weight functions on $Z$. Indeed, the extension of $ y \circ V $ and its derivatives are continuous on $ Y $, whence uniformly continous on the compact set $ B_R(0) $. Let us fix a sequence of increasing finite-dimensional, orthogonal projections $ \pi_n \to  \operatorname{id}_Y$ converging strongly to the identity: composing the extension of $ y \circ V $ with $ \pi_n $ yields a pointwise converging, equicontinous sequence of cylindrical function on $ B_R(0) $, which is -- up to a smoothing argument -- the desired assertion.

See also \cite[Theorem~5]{DoersekTeichmann2011} and \cite[Theorem~2.39]{Doersek2011phdthesis} for comparable arguments. In particular, this implies that Nemytskii operators are included in our setup if $Z$ is a Sobolev space of sufficiently smooth functions, see also \cite[Example~2.48]{Doersek2011phdthesis}.

This definition should also be compared to the form of the multiplicative noise suggested in \cite[Remark~2.3]{Debussche2011}.
	It is similar in spirit to the definition of $\mathcal{C}^{\varphi}_k(H;H)$, as there, $A$ is assumed to be a negative self-adjoint operator with a compact inverse.
	Hence, if we consider a single component of the noise, $x\mapsto\tilde{\sigma}( (-A)^{-1/4}x )$, with $\tilde{\sigma}\colon H\to H$ a $\mathrm{C}^3$-function with derivatives bounded up to order $3$, it satisfies our assumptions given above and hence lies in $\mathcal{C}^{\varphi}_3(H)$ with $\varphi_0(x):=(1+\lVert x\rVert_H^2)^{1/2}$ and $\varphi_j(x):=1$, $j\ge 1$.
\end{remark}

\begin{theorem}
	\label{thm:vectorfieldsBxC}
	Fix $k\ge 1$.
	Let $\psi=(\psi_i)_{i=0,\ldots,k}$ be a vector of D-admissible weight functions on $X$, and $\hat{\psi}=(\hat{\psi}_j)_{j=0,\ldots,k-1}$ a vector of D-admissible and $\varphi=(\varphi_j)_{j=0,\ldots,k-1}$ a vector of C-admissible weight functions on $Z$.
	Suppose \eqref{eq:vectorfields-phipsirelation}.

	Then, the Lie derivative $\mathcal{L}\colon\mathcal{C}^{\varphi}_{k-1}(Z;X)\times\mathcal{B}^{\psi}_{k}(X)\to\mathcal{B}^{\hat{\psi}}_{k-1}(Z)$ defined through
	\begin{equation}
		\mathcal{L}(V,f)(z)
		:=
		\mathcal{L}_V f(z)
		:=Df(z)(V(z))
	\end{equation}
	is a bilinear, bounded operator.
\end{theorem}
\begin{remark}
	Clearly, it is necessary that $V\in\mathcal{C}^{\varphi}_k(Z;X)$ if $\mathcal{L}_V f\in\mathcal{B}^{\hat{\psi}}_k(Z)$ is supposed to hold for $f\in\mathcal{B}^{\psi}_{k+1}(X)$ for a sufficiently large class of weight functions $\psi$.
	Indeed, choose $\psi_0(x):=\rho(\lVert x\rVert_X)$ with some increasing, left continuous and superlinear function $\rho$, and $\psi_j$ arbitrary D-admissible weight functions on $X$.
	Then, $f:=y\in\mathcal{B}^{\psi}_{k+1}(X)$ for all $y\in X^{*}$.
	Hence, $\mathcal{L}_V f(z)=y(V(z))$, and $y\circ V\in\mathcal{B}^{\hat{\psi}}_k(Z)$ implies that $V\in\mathcal{C}^{\varphi}_k(Z;X)$.
\end{remark}
\begin{proof}
	The claimed boundedness of $\mathcal{L}$ was remarked above, and follows straight away from \eqref{eq:vectorfields-phipsirelation}.

	Hence, we only need to prove that $\mathcal{L}_V f\in\mathcal{B}^{\hat{\psi}}_{k-1}(Z)$ for given $V\in\mathcal{C}^{\varphi}_{k-1}(Z;X)$ and $f=g(\langle\cdot,w_1\rangle,\cdots,\langle\cdot,w_n\rangle)\in\mathcal{A}(X)$; the result then follows from a density argument.
	Fix $\varepsilon>0$.
	We shall construct $g_{\varepsilon}\in\mathcal{A}(Z)$ such that $\lVert \mathcal{L}_V f - g_{\varepsilon} \rVert_{\hat{\psi},k}<C\varepsilon$ with some constant $C>0$ independent of $\varepsilon$.

	Choose a dual set of vectors $(\zeta_i)_{i=1,\ldots,n}\subset Z$ of $(w_i)_{i=1,\ldots,n}$, i.e., $\langle \zeta_i,w_j\rangle = \delta_{ij}$.
	Let $Z_n:=\lspan\left\{ \zeta_i\colon i=1,\cdots,n \right\}$, and define $\pi\colon X\to Z_n$ by $\pi x:=\sum_{i=1}^{n}\langle x,w_i\rangle \zeta_i$.
	Then, $f\circ\pi=f$, and
	\begin{equation}
		\mathcal{L}_V f(z)
		=
		\sum_{i=1}^{n}Df(z)(\zeta_i)\langle V(z), w_i\rangle.
	\end{equation}
	Clearly, $w_i\in X^{*}$, and thus by Definition \ref{def:vectorfield}, there exists $C_V:=\max_{i=1,\dots,n}C_{V,w_i}>0$ such that for all $R>0$, we can find $v^i_{R,\varepsilon}\in\mathcal{A}(Z)$ with $\lVert v^i_{R,\varepsilon}\rVert_{\varphi,k-1}\le C_V$ and
	\begin{equation}
		\lVert w_i\circ V - v^i_{R,\varepsilon}\rVert_{\mathrm{C}^{k-1}(B_R(0))} < \varepsilon,
	\end{equation}
	where $B_R(0)$ denotes the closed unit ball in $Z$.
	Setting $g_{\varepsilon}:=\sum_{i=1}^{n}Df(\cdot)(\zeta_i)v^i_{R,\varepsilon}\in\mathcal{A}(Z)$, it follows that with a constant $C_f>0$ independent of $R>0$,
	\begin{equation}
		\lVert \mathcal{L}_V f - g_{\varepsilon} \rVert_{\mathrm{C}^{k-1}(B_R(0))}
		< C_f\varepsilon.
	\end{equation}
	Choose $R_{\varepsilon}>0$ large enough such that $\psi_j(z)>\varepsilon^{-1}$ for $\lVert z\rVert_{Z}>R_{\varepsilon}$.
	This is possible as the embedding $Z\to X$ is continuous.
	Hence, as $f$ and all its derivatives are bounded,
	\begin{equation}
		\hat{\psi}_j(z)^{-1}\lVert D^j \mathcal{L}_V f(z)\rVert_{L_j(Z)}<C_f\varepsilon
		\quad\text{for $\lVert z\rVert_Z>R_{\varepsilon}$},
		\quad j=0,\dots,k-1,
	\end{equation}
	where $C_f$ is independent of $\varepsilon$.
	Furthermore,
	\begin{equation}
		\hat{\psi}_j(z)^{-1}\lVert D^j g_{\varepsilon}(z)\rVert_{L_j(Z)}
		\le 
		C_{f,V}\varepsilon
		\quad\text{for $\lVert z\rVert_Z>R_{\varepsilon}$},
		\quad j=0,\dots,k-1,
	\end{equation}
	where $C_{f,V}>0$ depends on $f$ and $V$, but not on $\varepsilon$ or $R_{\varepsilon}$.
	Plugging the results together proves the claim.
\end{proof}

Let us consider two special cases.
\begin{corollary}
  \label{cor:boundedvectorfields}
  Let $(H,\lVert\cdot\rVert_{H})$ be a Hilbert space, $(Z,\lVert\cdot\rVert_{Z})$ a continuously embedded Hilbert space.
  Define the D-admissible weight functions $\psi_j(x):=\cosh(\lVert x\rVert_H)$ on $H$ and $\hat{\psi}_j(x):=\cosh(\lVert x\rVert_Z)$ on $Z$ and the C-admissible weight functions $\varphi_j(x):=1$ on $Z$, $j\ge 0$.
  Then, for every $k\ge 0$, the mapping
  \begin{alignat}{2}{}
    \mathcal{L}\colon\mathcal{B}^{\psi}_k(X)\times\mathcal{C}^{\varphi}_{k-1}(Z;X) \to \mathcal{B}^{\hat{\psi}}_{k-1}(Z),
    \quad
    (f,V)\mapsto \mathcal{L}_V f,
  \end{alignat}
  given by $\mathcal{L}_V f(x):=Df(x)V(x)$, is bounded and bilinear.
\end{corollary}
\begin{remark}
  If $Z=H$, this has the simple interpretation that bounded vector fields map $\cosh$-weighted spaces into themselves.
\end{remark}
\begin{proof}
  This is straightforward from Theorem~\ref{thm:vectorfieldsBxC}, as the $\hat{\psi}_j$ defined there is only a multiple of $\hat{\psi}_j$ in this case.
\end{proof}

The following special case is very useful in the analysis of stochastic partial differential equations of Da Prato-Zabczyk type.
\begin{corollary}
  \label{cor:Lipschitzvectorfields}
  Let $(H,\lVert\cdot\rVert_{H})$ be a Hilbert space, $(Z,\lVert\cdot\rVert_{Z})$ a continuously embedded Hilbert space.
  Fix $n\in\mathbb{N}$.
  Define the D-admissible weight functions $\psi_j(x):=(1+\lVert x\rVert_H^2)^{(n-j)/2}$ on $H$ and $\hat{\psi}_j(x):=(1+\lVert x\rVert_Z^{2})^{(n-j)/2}$ on $Z$, $j=0,\dots,n-1$, and the C-admissible weight functions $\varphi_0(x):=(1+\lVert x\rVert_{Z}^2)^{1/2}$ and $\varphi_j(x):=1$ on $Z$, $j\in\mathbb{N}$.
  Then, for $k\le n-1$, the mapping
  \begin{alignat}{2}{}
    \mathcal{L}\colon\mathcal{B}^{\psi}_k(X)\times\mathcal{C}^{\varphi}_{k-1}(Z;X) \to \mathcal{B}^{\hat{\psi}}_{k-1}(Z),
    \quad
    (f,V)\mapsto \mathcal{L}_V f,
  \end{alignat}
  given by $\mathcal{L}_V f(x):=Df(x)V(x)$, is bounded and bilinear.
\end{corollary}
\begin{remark}
	\label{rem:LipschitzVF}
  This means that linearly bounded vector fields $Z\to X$ with bounded derivatives (hence also Lipschitz continuous) map polynomially bounded functions to polynomially bounded functions, with the same weights.
	In particular, if $A\colon\dom A\subset H\to H$ is a densely defined, closed operator, then $V_A\in\mathcal{C}^{\varphi}_k(\dom A;H)$ for all $k\ge 0$, where $\varphi$ is defined as in Corollary~\ref{cor:Lipschitzvectorfields}, $\dom A$ is endowed with the operator norm, and $V_A(x):=Ax$ for $x\in\dom A$.
\end{remark}
\begin{proof}
  Calculating
  \begin{alignat}{2}{}
    (1+\lVert x\rVert_{Z}^2)^{(n-1)/2}&(1+\lVert x\rVert_{Z}^2)^{1/2}
    +\sum_{i=0}^{j}\binom{j}{i}(1+\lVert x\rVert_{Z}^2)^{(n-i-1)/2} \notag \\
    &\le
    C\hat{\psi}_j(x),
  \end{alignat}
  the claim again follows from an application of Theorem~\ref{thm:vectorfieldsBxC}.
\end{proof}

\section{Stability of cubature schemes}
We shall now prove stability of cubature on Wiener space in the setting of weighted spaces.
Consider from now on the following setup.
Let on $[0,1]$ be given paths $(\omega^{(1)}_i)_{i=1,\dots,N}$, $\omega^{(1)}_i(s)=(\omega^{(1),j}_i(s))_{j=0,\dots,d}$, $\omega^{(1),0}_i(s)=s$, and weights $(\lambda_i)_{i=1,\dots,N}$ of a cubature on Wiener space of order $m\ge 1$ for a $d$-dimensional Brownian motion, i.e., for all multi-indices $(j_1,\dots,j_k)$ with $k+\#\left\{ i\colon j_i=0 \right\}\le m$ and a $d$-dimensional Brownian motion $(B^j_t)_{j=1,\dots,d, t\ge 0}$,
\begin{alignat}{2}{}
	\mathbb{E}&\left[ \idotsint_{0\le s_1\le \dots \le s_k\le 1}\circ\dd B^{j_1}_{s_1}\dots\circ\dd B^{j_k}_{s_k} \right] \\
	&=
	\sum_{i=1}^{N} \lambda_i
	\idotsint_{0\le s_1\le \dots \le s_k\le 1}\dd \omega_i^{(1),j_1}(s_1)\dots\circ\dd \omega_i^{(1),j_k}(s_k). \notag
\end{alignat}
Here, we have set $B^0_t:=t$ and $\circ\dd B^0_t:=\dd t$ for ease of notation.
For a general time interval $[0,\Delta t]$, we set 
\begin{equation}
	\omega^{(\Delta t),0}_i(s):=s
	\quad\text{and}\quad
	\omega^{(\Delta t),j}_i(s):=\sqrt{\Delta t}\omega^{(1),j}_i(s/\Delta t),
	\quad j=1,\dots,d,
\end{equation}
so that $(\omega^{(\Delta t)}_i)_{i=1,\dots,N}$ and $(\lambda_i)_{i=1,\dots,N}$ define a cubature formula on Wiener space of order $m$ on $[0,\Delta t]$.
The approximation of the Markov semigroup $(P_t)_{t\ge 0}$, given by $P_t f(x):=\mathbb{E}[f(X^x_t)]$ for a function $f\colon H\to\mathbb{R}$, where $(X^x_t)_{t\ge 0}$ solves the Stratonovich stochastic differential equation
\begin{equation}
  \dd X^{x}_{t}
  =
  \sum_{j=0}^{d}V_j(X^x_t)\circ\dd B^{j}_{t}, 
  \quad
  X^{x}_0 = x,
\end{equation}
on some state space $H$, then reads
\begin{equation}
	P_t f(x)
	\approx
	Q_{(t/n)}^n f(x),
\end{equation}
where the one step approximation operator is defined by
\begin{equation}
  Q_{(\Delta t)}f(x) := \sum_{i=1}^{N}\lambda_i f(X^{x}_{\Delta t}(\omega_i^{(\Delta t)})),
\end{equation}
with $X^{x}_{t}(\omega_i^{(\Delta t)})$ the solution of the problem
\begin{equation}
  \dd X^{x}_{s}(\omega^{(\Delta t)}_i)
  =
  \sum_{j=0}^{d}V_j(X^x_s(\omega^{(\Delta t)}_i))\dd\omega^{(\Delta t),j}_i(s),
  \quad
  X^{x}_0(\omega^{(\Delta t)}_i) = x.
\end{equation}
Under certain smoothness assumptions on the vector fields $V_j$, $j=0,\dots,d$, and the payoff $f$, we expect that
\begin{equation}
	\lvert P_f f(x) - Q_{(t/n)}^n f(x) \rvert
	\le Cn^{-(m-1)/2},
\end{equation}
where the constant $C>0$ can depend on $f$, $V_j$, $j=0,\dots,d$, and $x\in H$.
For the case $H$ finite-dimensional and $f$ and $V_j$ bounded and $\mathrm{C}^{\infty}$-bounded, $j=0,\dots,d$, it is known that $C$ depends on the supremum norms of $f$ and its derivatives, but not on $x\in H$, see \cite{LyonsVictoir2004}.
For more background on the method, see \cite{LyonsVictoir2004,CrisanGhazali2007,BayerTeichmann2008}.
An alternative approach can be found in \cite{Kusuoka2001,Kusuoka2004}.
Its implementation as a splitting method is given in \cite{NinomiyaVictoir2008}, see also \cite{NinomiyaNinomiya2009,Alfonsi2010,TanakaKohatsuHiga2009}.

Our strategy is as follows.
First, we consider the finite dimensional case.
Here, the analysis is straightforward.
Afterwards, we turn to the infinite dimensional setting.
Here, our aim is to prove stability for Da Prato-Zabczyk equations with pseudodissipative generator.
We prove first the auxiliary result in Theorem~\ref{thm:nonautinfdim-stability}, which might be of independent interest.
The method of the moving frame then yields first Theorem~\ref{thm:infdimgroup}, and the Sz\H{o}kefalvi-Nagy theorem allows us to conclude in Corollary~\ref{cor:stabilitycubaturesemigroupspde}.

\subsection{Finite dimensional state space}
Given a Stratonovich SDE on $\mathbb{R}^n$,
\begin{equation}
  \dd X^{x}_{t}
  =
  \sum_{j=0}^{d}V_j(X^x_t)\circ\dd B^{j}_{t}, 
  \quad
  X^{x}_0 = x,
\end{equation}
we let the local discretisation of $P_t f(x):=\mathbb{E}[f(X^x_t)]$ be defined by
\begin{equation}
  Q_{(\Delta t)}f(x) := \sum_{i=1}^{N}\lambda_i f(X^{x}_{\Delta t}(\omega_i^{(\Delta t)})),
\end{equation}
where $X^{x}_{t}(\omega_i^{(\Delta t)})$ is the solution of the problem
\begin{equation}
  \dd X^{x}_{s}(\omega^{(\Delta t)}_i)
  =
  \sum_{j=0}^{d}V_j(X^x_s(\omega^{(\Delta t)}_i))\dd\omega^{(\Delta t),j}_i(s),
  \quad
  X^{x}_0(\omega^{(\Delta t)}_i) = x.
\end{equation}
\begin{theorem}
  \label{thm:stabilityfinitedimensions}
  Let $\psi$ be an admissible weight function on $\mathbb{R}^{n}$, and assume that 
  \begin{equation}
    \lvert V_i V_j\psi(x) \rvert + \lvert V_i\psi(x) \rvert
    \le C\psi(x)
    \quad\text{for $i=0,\cdots,d$ and $j=1,\cdots,d$},
		\label{eq:stabilityfinitedimensions-assumption}
  \end{equation}
  where we require that all the necessary derivatives are well-defined.

	Then, there exists a constant $\tilde{C}>0$ independent of $\Delta t>0$ such that
  \begin{equation}
    Q_{(\Delta t)}\psi(x)
    \le
		\exp(\tilde{C}\Delta t)\psi(x).
  \end{equation}
\end{theorem}
\begin{proof}
  We define the intermediate operator
  \begin{equation}
    Q_{(\Delta t,s)}f(x) := \sum_{i=1}^{N}\lambda_i f(X^{x}_{s}(\omega_i^{(\Delta t)}))
    \quad\text{for $s\in[0,t]$}
  \end{equation}
  and note that $Q_{(\Delta t)}=Q_{(\Delta t,\Delta t)}$.
	The definition of the iteration step yields
  \begin{alignat}{2}
    \psi(X^{x}_{s}(\omega^{(\Delta t)}_{i}))
    &=
    \psi(x) + \sum_{j=0}^{d}\int_{0}^{s}V_j\psi(X^{x}_{r}(\omega^{(\Delta t)}_i))\dd\omega^{(\Delta t),j}_{i}(r) \\
    &=
    \psi(x) + \int_{0}^{s}V_0\psi(X^{x}_{r}(\omega^{(\Delta t)}_{i}))\dd r + \sum_{j=1}^{d}V_j\psi(x)\omega^{(\Delta t),j}_i(s) \notag \\
    &\phantom{=}+ \sum_{j=1}^{d}\sum_{k=0}^{d}\int_{0}^{s}\int_{0}^{r}V_k V_j\psi(X^{x}_{q}(\omega^{(\Delta t)}_i))\dd\omega^{(\Delta t),k}_i(q)\dd\omega^{(\Delta t),j}_i(r). \notag
  \end{alignat}
	By \eqref{eq:stabilityfinitedimensions-assumption},
  \begin{alignat}{2}
    \int_{0}^{s}V_0\psi(X^{x}_{r}(\omega^{(\Delta t)}_{i}))\dd r
    \le
    C\int_{0}^{s}\psi(X^{x}_{r}(\omega^{(\Delta t)}_{i}))\dd r.
  \end{alignat}
  Furthermore, as $\lvert\omega^{(\Delta t),j}_{i}(s)\rvert\le C(\Delta t)^{1/2}$ and $\lvert\frac{\partial}{\partial s}\omega^{(\Delta t),j}_{i}(s)\rvert\le C(\Delta t)^{-1/2}$, Fubini's theorem yields
  \begin{alignat}{2}
    \int_{0}^{s}\int_{0}^{r}&V_{k}V_{j}\psi(X^{x}_{q}(\omega^{(\Delta t)}_{i}))\dd\omega^{(\Delta t),k}_i(q)\dd\omega^{(\Delta t),j}_{i}(r) \notag \\
    &\le
    C\int_{0}^{s}\lvert \omega^{(\Delta t),j}_i(s)-\omega^{(\Delta t),j}_i(q) \rvert\psi(X^{x}_{q}(\omega^{(\Delta t)}_i))\left\lvert\frac{\partial}{\partial q}\omega^{(\Delta t),j}_i(q)\right\rvert\dd q \notag \\
    &\le
    C\int_{0}^{s}\psi(X^{x}_{q}(\omega^{(\Delta t)}_i))\dd q.
  \end{alignat}
  Thus, we see that
  \begin{alignat}{2}
    Q_{(\Delta t,s)}\psi(x)
    &=
    \sum_{i=1}^{N}\lambda_i \psi(X^{x}_{s}(\omega^{(\Delta t)}_{i})) \\
    &\le
    \psi(x) + \sum_{j=1}^{d}V_j\psi(x)\sum_{i=1}^{N}\lambda_i\omega^{(\Delta t),j}_i(s) + C\int_{0}^{s}Q_{(\Delta t,r)}\psi(x)\dd r. \notag
  \end{alignat}
  Defining $\alpha_{\Delta t,s}(x):=\sum_{j=1}^{d}V_j\psi(x)\sum_{i=1}^{N}\lambda_i\omega^{(\Delta t),j}_i(s)$, the Gronwall inequality yields
  \begin{equation}
    Q_{(\Delta t,s)}\psi(x)
    \le
    \psi(x) + \alpha_{\Delta t,s}(x) + \int_{0}^{s}\left( \psi(x)+\alpha_{\Delta t,r}(x) \right)C\exp(C(s-r))\dd r.
  \end{equation}
  Note that $\alpha_{\Delta t,\Delta t}(x)=0$ by the equality $\sum_{i=1}^{N}\lambda_i\omega^{(\Delta t),j}_i(\Delta t)=0$.
	Furthermore,
  \begin{equation}
    \alpha_{\Delta t,s}(x)
    \le
    C\sqrt{\Delta t}\psi(x)
    \le
		\frac{C}{2}(1+\Delta t)\psi(x)
		\le
		\frac{C}{2}\exp(\Delta t)\psi(x).
  \end{equation}
  This proves
  \begin{alignat}{2}{}
    Q_{(\Delta t)}\psi(x)
    &=
    Q_{(\Delta t,\Delta t)}\psi(x)
    \le
		\psi(x)\left( 1+\frac{C}{2}\exp(\Delta t)( \exp(C\Delta t)-1 ) \right)
    \notag \\ &
    \le
		\exp(\tilde{C}\Delta t)\psi(x),
  \end{alignat}
	where $\tilde{C}:=\max(C^2/2,C+1)$, which is the required estimate.
\end{proof}

\subsection{Time-dependent stochastic ordinary differential equations on Hilbert space}
Let $H$ be a Hilbert space, and consider the nonautonomous stochastic ordinary differential equation
\begin{equation}
  \label{eq:nonautinfdim}
  \dd X^x_t = \sum_{j=0}^{d}V_j(t,X^x_t)\circ\dd B^j_t,
  \quad
  X^x_0 = x,
\end{equation}
on $H$.
We define cubature approximations of \eqref{eq:nonautinfdim} by
\begin{equation}
  \dd X^{t,x}_s(\omega^{(\Delta t)}_i)
  =
  \sum_{j=0}^{d}V_j(t+s,X^{t,x}_s(\omega^{(\Delta t)}_i))\dd\omega^{(\Delta t),j}_i(s),
  \quad
  X^{t,x}_0 = x,
\end{equation}
and the approximation operator by
\begin{equation}
  Q^{t}_{(\Delta t)}f(x)
  :=
  \sum_{i=1}^{N}\lambda_i f(X^{t,x}_{\Delta t}(\omega^{(\Delta t)}_i)).
\end{equation}
\begin{definition}
  A cubature formula $( \omega^{(\Delta t)}_i, \lambda_i )_{i=1,\dots,N}$ is called \emph{symmetric} if for every $i\in\{1,\dots,N\}$, there exists some $i'\in\left\{ 1,\dots,N \right\}$ such that $\lambda_i=\lambda_{i'}$ and
  \begin{alignat}{2}{}
    \omega^{(\Delta t),j}_i(s)=-\omega^{(\Delta t),j}_{i'}(s)
    \quad\text{for all $s\in[0,\Delta t]$ and $j=1,\dots,d$}.
  \end{alignat}

  It is called \emph{weakly symmetric} if for $j=1,\dots,d$,
  \begin{equation}
    \sum_{i=1}^{N}\lambda_i\omega^{(\Delta t),j}_i(s) = 0
    \quad\text{for $s\in[0,\Delta t]$}.
  \end{equation}
\end{definition}
\begin{remark}
  Clearly, all symmetric cubature formulas are also weakly symmetric.
  Note that many known cubature formulas are actually symmetric.
  Moreover, a non-symmetric cubature formula can be made symmetric by adding the negatives of the paths with the same weights to it and dividing all weights by two.
  This will at most double the number of paths.
  Thus, if we use a cubature formula with a small number of paths in high dimensions, we can also find a symmetric cubature formula with this property.
\end{remark}

\begin{theorem}
  \label{thm:nonautinfdim-stability}
  Suppose that the cubature formula used in the definition of $Q^t_{(\Delta t)}$ is weakly symmetric.
  Let $\psi$ be an admissible weight function on $H$ and suppose %
    \begin{alignat}{2}{} %
    \label{eq:nonautinfdim-psiderivatives1} %
      \lVert D\psi(x)\rVert  &\le C(1+\lVert x\rVert^2)^{-1/2}\psi(x) %
      \quad\text{and}\\ %
    \label{eq:nonautinfdim-psiderivatives2} %
      \lVert D^2\psi(x)\rVert&\le C(1+\lVert x\rVert^2)^{-1}\psi(x) %
    \end{alignat} %
  with some constant $C>0$, 
  Furthermore, assume that for some constant $C>0$ independent of $t$,
  \begin{equation}
    \label{eq:nonautinfdim-Vjlinearlybounded}
    \lVert V_j(t,x)\rVert
    \le
    C( 1+\lVert x\rVert^2 )^{1/2}
    \quad\text{for $j=0,\dots,d$, $x\in X$ and $t\in[0,T]$},
  \end{equation}
  and that $x\mapsto V_j(t,x)$ is continuously differentiable with derivative bounded uniformly in $t\in[0,T]$ for $j=1,\dots,d$.

	Then, there exists a constant $\tilde{C}>0$ such that for all $t\in[0,T]$ and $\Delta t\in[0,T-t]$,
  \begin{equation}
    Q^t_{(\Delta t)}\psi(x)
    \le
		\exp(\tilde{C}t)\psi(x)
    \quad\text{for all $x\in H$}.
  \end{equation}
\end{theorem}
\begin{remark}
	The above result is remarkable as we do not assume that the vector fields $V_j$ are differentiable with respect to $t$.
	This is also the reason why we cannot simply apply Theorem~\ref{thm:stabilityfinitedimensions} to conclude.
\end{remark}
\begin{proof}
  Define the intermediate approximation for $s\in[0,\Delta t]$ by
  \begin{equation}
    Q^{t}_{(\Delta t,s)}f(x)
    :=
    \sum_{i=1}^{N}\lambda_i f(X^{t,x}_{s}(\omega^{(\Delta t)}_i)).
  \end{equation}
	As in the proof of Theorem~\ref{thm:stabilityfinitedimensions}, we note that $Q^{t}_{(\Delta t,\Delta t)}=Q^{t}_{(\Delta t)}$.
  For $0\le s\le \Delta t$,
  \begin{alignat}{2}{}
    &\psi(X^{t,x}_{s}(\omega^{(\Delta t)}_i))
    =
    \notag \\ & \quad
    \psi(x)
    +
    \sum_{j=0}^{d}\int_{0}^{s}D\psi(X^{t,x}_{r}(\omega^{(\Delta t)}_i))V_j(t+r,X^{t,x}_{r}(\omega^{(\Delta t)}_i))\dd\omega^{(\Delta t),j}_i(r).
  \end{alignat}
  Consider $g_j(r,x):=D\psi(x)V_j(t+r,x)$.
  Then, 
  \begin{alignat}{2}{}
    g_j(\rho,\,& X^{t,x}_{r}(\omega^{(\Delta t)}_i))
    =
    g_j(\rho,x)
    \notag \\ &
    +
    \sum_{k=0}^{d}\int_{0}^{r}D_x g_j(\rho,X^{t,x}_{q}(\omega^{(\Delta t)}_{i}))V_k(t+q,X^{t,x}_{q}(\omega^{(\Delta t)}_i))\dd\omega^{(\Delta t),k}_i(q).
  \end{alignat}
  From \eqref{eq:nonautinfdim-psiderivatives1}, \eqref{eq:nonautinfdim-psiderivatives2} and \eqref{eq:nonautinfdim-Vjlinearlybounded}, we obtain that for $0\le s\le \Delta t\le T$,
  \begin{alignat}{2}{}
    \lvert g_0(r,x)\rvert
    &=
    \lvert D\psi(x)V_0(t+r,x) \rvert
    \le
    C\lVert D\psi(x)\rVert\cdot\lVert V_0(t+r,x)\rVert \notag \\
    &\le
    C\psi(x).
  \end{alignat}
  We argue in a similar manner for $D_x g_j(r,x)V_{k}(t+q,x)$, $j=1,\dots,d$, $k=0,\dots,d$, to obtain that for $0\le q\le r\le \Delta t$,
  \begin{alignat}{2}{}
    \lvert D_x g_j(r,x)V_k(t+q,x) \rvert
    &
    = \lvert D^2\psi(x)( V_j(t+r,x), V_k(t+q,x) ) 
    \notag \\ &\phantom{=\lvert}
    + D\psi(x)D_xV_j(t+r,x)V_k(t+q,x) \rvert \notag \\
    &\le C\psi(x).
  \end{alignat}
  An application of Fubini's theorem just as in the proof of Theorem~\ref{thm:stabilityfinitedimensions} gives
  \begin{alignat}{2}{}
    &\psi(X^{t,x}_s(\omega^{(\Delta t)}_i))
    =
    \psi(x)
    +
    \int_{0}^{s}g_0(r,X^{t,x}_{r}(\omega^{(\Delta t)}_i))\dd r 
    +
    \sum_{j=1}^{d}\int_{0}^{s}g_j(r,x)\dd\omega^{(\Delta t),j}_i(r) \notag \\
    &\phantom{=}+
    \sum_{j=1}^{d}\sum_{k=0}^{d}\int_{0}^{s}\int_{0}^{r}D_x g_j(r,X^{t,x}_{q}(\omega^{(\Delta t)}_i))V_k(t+q,X^{t,x}_{q}(\omega^{(\Delta t)}_{i}))\dd\omega^{(\Delta t),k}_i(q)\dd\omega^{(\Delta t),j}_i(r) \notag \\
    &\le
    \psi(x) + C\int_{0}^{s}\psi(X^{x}_{r}(\omega^{(\Delta t)}_i))\dd r + \sum_{j=1}^{d}\int_{0}^{s}g_{j}(r,x)\dd \omega^{(\Delta t),j}_i(r),
  \end{alignat}
  where we apply that $\Delta t\le T$.
  As from the weak symmetry of the cubature paths,
  \begin{alignat}{2}{}
    \sum_{i=1}^{N}\lambda_i\sum_{j=1}^{d}\int_{0}^{s}g_j(r,x)\dd\omega^{(\Delta t),j}_i(r)
    &=
    \sum_{j=1}^{d}\int_{0}^{s}g_j(r,x)\dd\biggl( \sum_{i=1}^{N}\lambda_i\omega^{(\Delta t),j}_i(r) \biggr) \notag \\
    &=
    0,
  \end{alignat}
  we obtain
  \begin{equation}
    Q_{(\Delta t,s)}\psi(x)
    \le
    \psi(x) + C\int_{0}^{s}Q_{(\Delta t,r)}\psi(x)\dd r.
  \end{equation}
  An application of Gronwall's lemma yields $Q_{(\Delta t)}\psi(x) \le \exp(C\Delta t)\psi(x)$, which proves the result.
\end{proof}
\begin{remark}
  It is clear that the given assumptions on the vector fields and the weight function are not the only ones possible.
  Instead, we could also require the vector fields to be bounded uniformly in $t\in[0,T]$, and allow the weight function to satisfy $\lVert D\psi(x)\rVert + \lVert D^2\psi(x)\rVert\le C\psi(x)$.
  While the situation above corresponds to polynomially growing weight functions and linearly bounded vector fields, this variant corresponds to exponentially growing weight functions and bounded vector fields, see also Corollaries~\ref{cor:Lipschitzvectorfields} and \ref{cor:boundedvectorfields}.

  Such an approach might be more appropriate when dealing with exponentials of stochastic processes such as L\'evy processes, which are ubiquitous in applications in mathematical finance, as they ensure nonnegativity in a simple manner and allow us to work on the natural scale of the problem.
\end{remark}

\subsection{Da Prato-Zabczyk equations}
Suppose now that
\begin{equation}
  \dd X^x_t = AX^x_t\dd t + \sum_{j=0}^{d}V_j(X^x_t)\circ\dd B^j_t,
  \quad
  X^x_0 = x,
\end{equation}
is a stochastic partial differential equation of Da Prato-Zabczyk type on some Hilbert space $H$, see \cite{DaPratoZabczyk1996,DaPratoZabczyk2002} for a comprehensive exposition of the theory of such equations.
Here, solutions are understood in the mild sense,
\begin{equation}
  X^x_t
  =
  \exp(tA)x + \sum_{j=0}^{d}\int_{0}^{t}\exp( (t-s)A )V_j(X^x_t)\circ\dd B^j_s,
\end{equation}
and we also define the cubature discretisations in the mild sense,
\begin{equation}
  X^x_t(\omega^{(\Delta t)}_i)
  =
  \exp(tA)x + \sum_{j=0}^{d}\int_{0}^{t}\exp( (t-s)A )V_j(X^x_t(\omega^{(\Delta t)}_i))\dd\omega^{(\Delta t),j}_i(s).
\end{equation}
Again, the approximation of the Markov semigroup $P_t f(x):=\mathbb{E}[f(X^x_t)]$ is given by
\begin{equation}
  Q_{(\Delta t)}f(x)
  :=
  \sum_{i=1}^{N}\lambda_i f(X^{x}_{\Delta t}(\omega^{(\Delta t)}_i)).
\end{equation}

\begin{theorem}
  \label{thm:infdimgroup}
  Suppose that $A$ is the generator of a group $S_t=\exp(tA)$, $t\in\mathbb{R}$, and that the cubature formula used in the definition of $Q_{(\Delta t)}$ is weakly symmetric.
	Let $\psi$ be an admissible weight function on $H$. 
  With some constant $C>0$, let $\psi(S_tx)\le\exp(Ct)\psi(x)$ for all $x\in H$ and $t>0$, and
  \begin{subequations}
    \label{eq:assumptionpsiderivatives}
    \begin{alignat}{2}{}
      \lVert D\psi(x)\rVert   &\le C(1+\lVert x\rVert^2)^{-1/2}\psi(x)
      \quad\text{and} \\
      \lVert D^2\psi(x)\rVert &\le C(1+\lVert x\rVert^2)^{-1}\psi(x).
    \end{alignat}
  \end{subequations}
  Furthermore, assume that
  \begin{equation}
    \label{eq:Vjlinearlybounded}
    \lVert V_j(x)\rVert
    \le
    C( 1+\lVert x\rVert^2 )^{1/2}
    \quad\text{for $j=0,\dots,d$},
  \end{equation}
  and that $V_j$ is continuously differentiable with bounded derivative for $j=1,\dots,d$.

  Then, for any $T>0$, there exists a constant $C>0$ such that for every $\Delta t\in[0,T]$, the operator $Q_{(\Delta t)}$ satisfies
  \begin{equation}
    Q_{(\Delta t)}\psi(x)
    \le
    \exp(C\Delta t)\psi(x)
    \quad\text{for all $x\in H$}.
  \end{equation}
\end{theorem}
\begin{proof}
  We apply the method of the moving frame from \cite{Teichmann2009}.
  This yields that $X^x_t=S_t Y^x_t$, where $(Y^x_t)_{t\ge 0}$ satisfies the Hilbert space stochastic ordinary differential equation
  \begin{equation}
    \dd Y^x_t = \sum_{j=0}^{d}\tilde{V}_j(t,Y^x_t)\circ\dd B^j_t,
    \quad
    Y^y_0 = y,
  \end{equation}
  with $\tilde{V}_j(t,y)=S_{-t}V_j(S_ty)$.
  Thus, rewriting the cubature discretisations of $(X^{x}_{t})_{t\ge 0}$ using $(Y^{x}_{t})_{t\ge 0}$,
  \begin{equation}
    \dd Y^{x}_s(\omega^{(\Delta t)}_i)
    =
    \sum_{j=0}^{d}\tilde{V}_j(s,Y^{x}_s(\omega^{(\Delta t)}_i))\dd\omega^{(\Delta t),j}_i(s),
  \end{equation}
  we see that, if we define
  \begin{equation}
    \tilde{Q}_{(\Delta t)}f(y) := \sum_{i=1}^{N}\lambda_i f(Y^{y}_{\Delta t}(\omega^{(\Delta t)}_i))
  \end{equation}
  for $f\colon H\to\mathbb{R}$, then $Q_{(\Delta t)}h(x)=\tilde{Q}_{(\Delta t)}g(x)$, where $g(y):=h(S_{\Delta t}y)$.
  In particular,
  \begin{equation}
    Q_{(\Delta t)}\psi(x)
    =
    \tilde{Q}_{(\Delta t)}(\psi\circ S_{\Delta t})(x)
    \le
    \exp(C\Delta t)\tilde{Q}_{(\Delta t)}\psi(x),
  \end{equation}
  where we apply the assumptions on $\psi$ and the positivity of $\tilde{Q}_{(\Delta t)}$.

  But now, we are in the situation of Theorem~\ref{thm:nonautinfdim-stability}:
  The estimates for $\psi$ are clear by assumption, and for $\tilde{V}_j(s,y)$, we note that, as $s\in[0,T]$,
  \begin{equation}
    \lVert \tilde{V}_j(s,y) \rVert
    =
    \lVert S_{-s}V_j(S_s y)\rVert
    \le
    C(1+\lVert x\rVert^{2})^{1/2}
    \quad\text{for $j=0,\dots,d$}
  \end{equation}
  and
  \begin{equation}
    \lVert D_y\tilde{V}_j(s,y)\rVert
    =
    \lVert S_{-s}D_y V_j(S_s y)S_s\rVert
    \le C
    \quad\text{for $j=1,\dots,d$}.
  \end{equation}
  An appeal to Theorem~\ref{thm:nonautinfdim-stability} yields
  \begin{equation}
    \tilde{Q}_{(\Delta t)}\psi(x)
    \le
    \exp(C\Delta t)\psi(x),
  \end{equation}
  and the result follows.
\end{proof}
The Sz\H{o}kefalvi-Nagy theorem now allows us to obtain a corresponding result for pseudocontractive semigroups.
\begin{corollary}
  \label{cor:stabilitycubaturesemigroupspde}
  Suppose that $A$ is the generator of a semigroup of pseudocontractions $S_t=\exp(tA)$, $t\ge 0$.
  Let $\psi(x)=\rho(\lVert x\rVert^2)$ with some increasing and left continuous function $\rho\colon[0,\infty)\to(0,\infty)$ (see also \cite[Example~4.1]{DoersekTeichmann2010}) which satisfies $\rho(Cu)\le C\rho(u)$ for all $u\ge 0$ and $C>0$, and which is twice differentiable and satisfies
  \begin{equation}
    \rho'(u)\le C(1+u)^{-1}\rho(u)
    \quad\text{and}\quad
    \rho''(u)\le C(1+u)^{-2}\rho(u).
  \end{equation}
  Furthermore, assume that $\lVert V_j(x)\rVert \le C( 1+\lVert x\rVert^2 )^{1/2}$ for $j=0,\dots,d$, and that $V_j$ is continuously differentiable with bounded derivative for $j=1,\dots,d$.

  Then, for any $T>0$, there exists a constant $C>0$ such that for every $\Delta t\in[0,T]$, the operator $Q_{(\Delta t)}$ satisfies
  \begin{equation}
    Q_{(\Delta t)}\psi(x)
    \le
    \exp(C\Delta t)\psi(x)
    \quad\text{for all $x\in H$}.
  \end{equation}
\end{corollary}
\begin{proof}
  Assume without loss of generality that $(S_t)_{t\ge 0}$ is a semigroup of contractions.
  By the Sz\H{o}kefalvi-Nagy theorem \cite[p.~452, Th\'eor\`eme~IV]{RieszSzNagy1955}, we see that we can find a Hilbert space $(\mathcal{H},\lVert\cdot\rVert_{\mathcal{H}})$ containing $H$ as a closed subspace and a strongly continuous group $(\mathcal{S}_t)_{t\in\mathbb{R}}$ of unitary mappings such that $S_t=\pi\mathcal{S}_t$, where $\pi\colon\mathcal{H}\to H$ is the orthogonal projection.

  Define $\psi_{\mathcal{H}}(y):=\rho(\lVert y\rVert_{\mathcal{H}}^2)$ and $V^{\mathcal{H}}_j(y):=V_j(\pi y)$, then it is easy to see that the assumptions of Theorem~\ref{thm:infdimgroup} are satisfied.
  The results of \cite{Teichmann2009} prove that the solution of
  \begin{equation}
    X^{\mathcal{H},y}_t = \mathcal{S}_t y + \sum_{j=0}^{d}\int_{0}^{t}\mathcal{S}_{t-s}V^{\mathcal{H}}_j(X^{\mathcal{H},y}_s)\circ\dd B^{j}_s
  \end{equation}
  satisfies $X^x_t=\pi X^{\mathcal{H},x}_t$, and similarly for the cubature approximations.
  Setting
  \begin{equation}
    Q^{\mathcal{H}}_{(\Delta t)}f(y) := \sum_{i=1}^{N}\lambda_i f(X^{\mathcal{H},y}_{\Delta t}(\omega^{(\Delta t)}_i)),
  \end{equation}
  Theorem~\ref{thm:infdimgroup} yields that $Q^{\mathcal{H}}_{(\Delta t)}\psi_{\mathcal{H}}(y)\le\exp(C\Delta t)\psi_{\mathcal{H}}(y)$, and from $\psi_{\mathcal{H}}(x)=\psi(x)$ for $x\in H$ we obtain that for $x\in H$,
  \begin{alignat}{2}{}
    Q_{(\Delta t)}\psi(x)
    &=
    \sum_{i=1}^{N}\lambda_i\rho(\lVert\pi X^{\mathcal{H},x}_{\Delta t}(\omega^{(\Delta t)}_i)\rVert^2) 
    \le
    \sum_{i=1}^{N}\lambda_i\rho(\lVert X^{\mathcal{H},x}_{\Delta t}(\omega^{(\Delta t)}_i)\rVert_{\mathcal{H}}^2)
    \notag \\&
    =
    Q^{\mathcal{H}}_{(\Delta t)}\psi_{\mathcal{H}}(x) 
    \le
    \exp(C\Delta t)\psi_{\mathcal{H}}(x)
    =
    \exp(C\Delta t)\psi(x).
  \end{alignat}
  The result is thus proved.
\end{proof}

\section{Convergence estimates of cubature schemes}
We are now ready to prove rates of convergence for cubature on Wiener space on weighted spaces.
We shall only prove these results in the infinite-dimensional setting; corresponding results in finite dimensions are obtained in a similar manner.

Let $H$ be a Hilbert space and $A$ the infinitesimal generator of a strongly continuous semigroup of pseudocontractions on $H$.
Fix $\ell_0\in\mathbb{N}$.
For $\ell=0,\dots,\ell_0$, let $H_{\ell}$ be subspaces of $H$ endowed with Hilbert norms $\lVert\cdot\rVert_{H_{\ell}}$, $H_0=H$, such that for $\ell=0,\dots,\ell_0-1$, $H_{\ell+1}\subset H_{\ell}$ and $A\colon H_{\ell+1}\to H_{\ell}$ is a bounded linear operator.
On $H_{\ell}$, we define D-admissible weight functions
\begin{equation}
  \psi_{\ell}^{s}(x) := \left( 1+\lVert x\rVert_{H_{\ell}}^2 \right)^{s/2},
  \quad s\ge 1, \quad \ell=0,\dots,\ell_0,
  \quad
  \psi^{s}:=\psi_{0}^{s},
\end{equation}
and the functions
\begin{equation}
  \varphi_{\ell,0}(x) := \left( 1+\lVert x\rVert_{H_{\ell}}^2 \right)^{1/2},
  \quad
  \varphi_{\ell,j}(x) := 1, \quad j\ge 1.
\end{equation}
Define the vectors of weight functions $\psi_{\ell}^{(n)}:=(\psi_{\ell}^{n-j})_{j=0,\dots,k}$, $k<n$, and $\varphi_{\ell}:=(\varphi_{\ell,j})_{j=0,\dots,k}$.
\begin{assumption}
	The vector fields satisfy
	\begin{alignat}{2}{}
		V_j\in\mathcal{C}^{\varphi_{\ell}}_k( H_{\ell}, H_\ell ) \quad\text{for $j=0,\dots,d$ and $\ell=0,\dots,\ell_0$}.
	\end{alignat}
\end{assumption}
Remark~\ref{rem:LipschitzVF} shows that $A\in\mathcal{C}^{\varphi_{\ell+1}}_k(H_{\ell+1},H_{\ell})$ for $\ell=0,\dots,\ell_0-1$.
For $x\in H_{\ell}$, $\ell=0,\dots,\ell_0$, we can then consider the Da Prato-Zabczyk equation
\begin{equation}
  \label{eq:stochTaylor-dpzeq}
  \dd X^x_t = AX^x_t\dd t + \sum_{j=0}^{d}V_j(X^x_t)\circ\dd B^j_t,
  \quad
  X^x_0 = x,
\end{equation}
on $H_{\ell}$.
As the assumptions on the vector fields $V_j$ essentially mean that they are Lipschitz continuous with bounded derivatives, all these equations have unique solutions, agreeing if we vary $\ell$ for sufficiently smooth initial conditions.

\begin{assumption}
	The Markov semigroup $(P_t)_{t\ge 0}$, $P_t f(x):=\mathbb{E}[f(X^x_t)]$, is strongly continuous on $\mathcal{B}^{\psi_{\ell}^{n}}( H_{\ell} )$ for all $n\in\mathbb{N}$ and $\ell=0,\dots,\ell_0$.
	For some $k_0\in\mathbb{N}$, $P_t$ is a bounded map from $\mathcal{B}^{\psi_{\ell}^{(n)}}_k( H_{\ell} )$ into itself for $k=0,\dots,k_0$ and $n\in\mathbb{N}$, $n>k$, with norm bounded uniformly in $t\in[0,T]$ for every $T>0$.
\end{assumption}
See also \cite[Section~5, Lemma~7.19]{DoersekTeichmann2010} for sufficient conditions for these assumptions.
\subsection{Taylor expansion of stochastic partial differential equations}
\begin{theorem}
  \label{thm:representationG1}
  Let $\ell=1,\dots,\ell_0$.
  Consider the strongly continuous semigroup $(P_t)_{t\ge 0}$ on the space $\mathcal{B}^{\psi_{\ell}^{n}}( H_{\ell} )$ with $n\ge 4$.
  Denote its generator by $(\mathcal{G},\dom\mathcal{G})$.

  Then, $\mathcal{B}^{\psi_{\ell-1}^{(n)}}_2( H_{\ell-1} )\subset\dom\mathcal{G}$, and
  \begin{alignat}{2}{}
    \label{eq:representationG}
    \mathcal{G}f(x)
    =
		Df(x)(Ax) &+ \mathcal{L}_{V_0}f(x) + \frac{1}{2}\sum_{j=1}^{d}\mathcal{L}_{V_j}^2 f(x) \\
    &\quad\text{for $f\in\mathcal{B}^{\psi_{\ell-1}^{(n)}}_2( H_{\ell-1} )$ and $x\in H_{\ell}$}. \notag 
  \end{alignat}
\end{theorem}
\begin{proof}
  By the It\^o formula \cite[Theorem~7.2.1]{DaPratoZabczyk2002}, it follows that for $f\in\mathcal{A}( H_{\ell-1} )$, we have $f\in\dom\mathcal{G}$ and $f$ satisfies \eqref{eq:representationG}.
  Corollary~\ref{cor:Lipschitzvectorfields} shows that the right hand side of \eqref{eq:representationG} is a continuous linear operator $\mathcal{B}^{\psi_{\ell-1}^{(n)}}_{2}( H_{\ell-1} )\to\mathcal{B}^{\psi_{\ell}^{n}}( H_{\ell} )$.
  The closedness of $\mathcal{G}$ proves the claim.
\end{proof}
The next result follows directly from Corollary~\ref{cor:Lipschitzvectorfields}, together with the explicit representation in \eqref{eq:representationG}.
\begin{corollary}
  Let $k\ge 0$.
  Under the assumptions of Theorem~\ref{thm:representationG1}, the infinitesimal generator $\mathcal{G}$ satisfies the mapping property
  \begin{equation}
    \mathcal{G}\colon\mathcal{B}^{\psi_{\ell-1}^{(n)}}_{k+2}( H_{\ell-1} ) \to \mathcal{B}^{\psi_{\ell}^{(n)}}_{k}( H_{\ell} ),
    \quad
    \ell=1,\dots,\ell_0.
  \end{equation}
\end{corollary}
Induction now yields:
\begin{corollary}
  Let $j=\ell,\dots,\ell_0$.
  Under the assumptions of Theorem~\ref{thm:representationG1}, the powers of the infinitesimal generator $\mathcal{G}$ satisfy
  \begin{equation}
    \mathcal{G}^j\colon\mathcal{B}^{\psi_{\ell-j}^{(n)}}_{k+2j}( H_{\ell-j} ) \to \mathcal{B}^{\psi_{\ell}^{(n)}}_{k}( H_{\ell} ).
  \end{equation}
  They are given explicitly by taking the powers of \eqref{eq:representationG}.
\end{corollary}
This allows us to obtain a Taylor expansion of $P_t f$ for smooth enough $f$, which we will compare to the Taylor expansion of cubature approximations.
\begin{corollary}
  \label{cor:localexpansionspde}
  Let $f\in\mathcal{B}^{\psi_{\ell-(k+1)}^{(n)}}_{2(k+1)}( H_{\ell-(k+1)} )$, $k+1\le\ell\le\ell_0$, $n\ge 2(k+2)$.

  Then,
  \begin{equation}
    P_t f = \sum_{j=0}^{k}\frac{t^j}{j!}\mathcal{G}^j f + t^{k+1}R_{t,k}f,
  \end{equation}
  where the linear operator $R_{t,k}\colon\mathcal{B}^{\psi_{\ell-(k+1)}^{(n)}}_{2(k+1)}( H_{\ell-(k+1)} )\to\mathcal{B}^{\psi_{\ell}^{n}}( H_{\ell} )$ satisfies
  \begin{equation}
    \lVert R_{t,k}f\rVert_{\psi_{\ell}^{n}}
    \le
    C_T \lVert f\rVert_{\psi_{\ell-(k+1)}^{(n)},2(k+1)}
    \quad\text{for $t\in[0,T]$}
  \end{equation}
  for a constant $C_T>0$ independent of $f$.
\end{corollary}

\subsection{Taylor expansion of cubature approximations}
For a multiindex $\alpha=(i_1,\dots,i_k)$, we define $\deg(\alpha):=k+\#\left\{ j=1,\dots,k\colon i_j=0 \right\}$.
The empty multiindex is denoted by $\emptyset$, corresponds to $k=0$, and satisfies $\deg(\emptyset)=0$.
We set
\begin{equation}
	\mathcal{A}_m:=\left\{ \alpha\colon \deg(\alpha)\le m \right\}
	\quad\text{and}\quad
	\mathcal{A}_m^{*}:=\mathcal{A}_m\setminus\left\{ \emptyset,(0) \right\}.
\end{equation}
\begin{theorem}
  \label{thm:localexpansioncubature}
  Assume that the cubature formula
  is of order $m=2k+1$.
  For $f\in\mathcal{B}^{\psi_{\ell-(k+1)}^{(n)}}_{2(k+1)}( H_{\ell-(k+1)} )$, $k+1\le\ell\le\ell_0$, $n\ge 2(k+2)$,
  \begin{equation}
    Q_{(\Delta t)}f
    =
    \sum_{j=0}^{k}\frac{(\Delta t)^j}{j!}\mathcal{G}^{j} f + (\Delta t)^{k+1}\hat{R}_{\Delta t,k}f,
  \end{equation}
  where the linear operator $\hat{R}_{\Delta t,k}\colon\mathcal{B}^{\psi_{\ell-(k+1)}^{(n)}}_{2(k+1)}( H_{\ell-(k+1)} )\to\mathcal{B}^{\psi_{\ell}^{n}}( H_{\ell} )$ satisfies
  \begin{equation}
    \lVert \hat{R}_{\Delta t,k}f\rVert_{\psi_{\ell}^{n}}
    \le
    C_T \lVert f\rVert_{\psi_{\ell-(k+1)}^{(n)},2(k+1)}
    \quad\text{for $\Delta t\in[0,T]$}
  \end{equation}
  for a constant $C_T>0$ independent of $f$.
\end{theorem}
\begin{proof}
  Under the assumptions on the vector fields, we can easily see that for every $f\in\mathcal{A}(H_{\ell-(k+1)})$, we have the Taylor expansion
  \begin{alignat}{2}{}
    f(X^x_{\Delta t}&(\omega^{(\Delta t)}_i)) \\
    &=
    \sum_{(i_1,\dots,i_k)\in\mathcal{A}_m}V_{i_1}\dots V_{i_k}f(x) I^{(i_1,\dots,i_k)}_{\Delta t}(\omega^{(\Delta t)}_i) + \hat{R}^{i}_{\Delta t,k}f(x), \notag
  \end{alignat}
  where we define the iterated integrals by
  \begin{alignat}{2}{}
    I^{(i_1,\dots,i_k)}_{\Delta t}&(\omega_i^{(\Delta t)},g)
    \\ &
    :=
    \int_{0<t_1<\dots <t_k<\Delta t}g(X^x_{t_1}(\omega^{(\Delta t)}_i))\dd\omega^{(t),i_1}_{i}(t_1)\dots\dd\omega^{(t),i_k}_{i}(t_k),
    \notag 
  \end{alignat}
  $I^{(i_1,\dots,i_k)}_{\Delta t}(\omega^{(\Delta t)}_{i}):=I^{(i_1,\dots,i_k)}_{\Delta t}(\omega^{(\Delta t)}_{i},1)$, the remainder term $\hat{R}^{i}_{\Delta t,k}f$ satisfies
  \begin{equation}
    \hat{R}^{i}_{\Delta t,k}f(x)
    =
    \sum_{\substack{(i_1,\dots,i_k)\in\mathcal{A}_m\\(i_0,i_1,\dots,i_k)\notin\mathcal{A}_{m}}}
    I^{(i_0,\dots,i_k)}_{\Delta t}(\omega^{(\Delta t)}_i,f_{(i_0,\dots,i_k)}),
  \end{equation}
  and we set $\beta_0(x):=Ax+V_0(x)$, $\beta_j(x):=V_j(x)$, $j=1,\dots,d$, and $f_{(i_0,\dots,i_k)}:=\beta_{i_0}\dots\beta_{i_k}f$, $(i_0,\dots,i_k)\in\left\{ 0,\dots,d \right\}^{k+1}$.
  Summing up, it is easy to see by the scaling of the cubature paths that we can find a remainder term $(\Delta t)^{k+1}$ as in the claim of the theorem with the correct estimates.
  To see that the initial terms have the form given, we use the order $2k+1$ of the cubature and the explicit formula of $\mathcal{G}$ from Theorem~\ref{thm:representationG1}.
  A density argument proves the result.
\end{proof}

\subsection{The rate of convergence}
We can now present our main result.
\begin{theorem}
  For $f\in\mathcal{B}^{\psi_{\ell-(k+1)}^{(n)}}_{2(k+1)}( H_{\ell-(k+1)} )$, $k+1\le\ell\le\ell_0$, $n>2(k+1)$, $2(k+1)\le k_0$,
  \begin{equation}
    \lVert P_T f - Q_{(T/n)}^n f\rVert_{\psi_{\ell}^{n}}
    \le
    C_T n^{-k}\lVert f\rVert_{\psi_{\ell-(k+1)}^{(n)},2(k+1)}
  \end{equation}
  with a constant $C_T$ independent of $f$.
\end{theorem}
\begin{proof}
  The local estimate follows from a combination of Corollary~\ref{cor:localexpansionspde} and Theorem~\ref{thm:localexpansioncubature}.
  The stability of $Q_{(T/n)}$ from Corollary~\ref{cor:stabilitycubaturesemigroupspde} and the assumed invariance of $\mathcal{B}^{\psi_{\ell-(k+1)}^{(n)}}_{2(k+1)}( H_{\ell-(k+1)} )$ with respect to $P_t$ prove the claim.
\end{proof}

\begin{example}
	Let $H_i=\mathbb{R}^N$ be finite-dimensional, and assume $n\ge 5$; in the finite-dimensional setting, we do not need to consider subspaces of the state space.
	Then, $f\in\mathcal{B}^{\psi^{(n)}}_{k}(H)$ for all $k\ge 0$, where $f(y)=y_i^m$, $m=1,\dots,4$.
	This implies that not only the expected value and the variance, but also the skewness and kurtosis are accurately computed by our scheme.
	Similarly, mixed moments are determined to high accuracy, and if $n$ is even larger, this also holds true for higher order moments.
	Such a property is very useful in risk management, where high precision in higher moments means an accurate evaluation of risk.
	Similar observations were made in \cite{Alfonsi2010,TanakaKohatsuHiga2009}.

	To illustrate this practically relevant behavior, we consider the Heston model, i.e., $(X^x_t,V^v_t)_{t\ge 0}$ solves the stochastic differential equation
	\begin{alignat}{2}{}
		\dd X^x_t &= \mu\dd t + \sqrt{V^v_t}\dd B^1_t, & \quad X^x_0 &= x, \\
		\dd V^v_t &= \kappa(\theta-V^v_t)\dd t + \beta\sqrt{V^v_t}\dd( \rho\dd B^1_t + \sqrt{1-\rho^2}\dd B^2_t ), & V^v_0 &= v,
	\end{alignat}
	and the stock price is given by $S_t:=\exp(X^x_t)$, see, e.g., \cite{BayerFrizLoeffen2010}.
	$V^v_t$ is the squared stochastic volatility.
	As $(X^x_t,V^v_t)_{t\ge 0}$ is a polynomial process, analytical formulas for the moments are available.
	For our simulation, we choose the parameters $\mu=.02$, $\kappa=5.$, $\theta=.09$, $\beta=.6$, and $\rho=-.8$.
	The starting values are chosen to be $x=\log(9.)$ and $v=.0625$, such that $S_0=9$.
	We are interested in finding the mean $m:=\mathbb{E}[X^x_t]$, the variance $\mathrm{var}:=\mathbb{E}[(X^x_t-m)^2]$, the skewness $\mathrm{skew}:=\mathrm{var}^{-3/2}\mathbb{E}[(X^x_t-m)^3]$, and the kurtosis $\mathrm{kurt}:=\mathrm{var}^{-2}\mathbb{E}[(X^x_t-m)^4]$.
	With the parameters given above, the exact values are found to be
	\begin{subequations}
		\begin{align}
			m &= 2.192936688809, & \mathrm{var} &= 0.019329503330, \\
			\mathrm{skew} &= -0.885007761283, \quad\text{and} & \mathrm{kurt} &= 4.321997672912.
		\end{align}
	\end{subequations}

	In our numerical simulation, we choose the cubature paths by applying a splitting of Ninomiya-Victoir type, where the normal random variables are replaced by Gauss-Hermite quadrature, see also \cite{TanakaKohatsuHiga2009}, where this splitting is also considered.
	For two-dimensional Brownian motion, this implies that we have $2q$ paths per interval, and that for $i=1,\dots,q$,
	\begin{alignat}{2}{}
		\omega^{(1)}_i(s) &:=
		\begin{cases}
			\begin{pmatrix} 3s, 0, 0 \end{pmatrix}^T, & s\in[0,1/3], \\
			\begin{pmatrix} 1, 3\xi_{i,1}(s-1/3), 0 \end{pmatrix}^T, & s\in[1/3,2/3], \\
			\begin{pmatrix} 1, \xi_{i,1}, 3\xi_{i,2}(s-2/3) \end{pmatrix}^T, & s\in[2/3,1],
		\end{cases}
		\quad\text{and} \\
		\omega^{(1)}_{q+i}(s) &:=
		\begin{cases}
			\begin{pmatrix} 0, 0, 3\xi_{i,2}s \end{pmatrix}^T, & s\in[0,1/3], \\
			\begin{pmatrix} 0, 3\xi_{i,1}(s-1/3), \xi_{i,2} \end{pmatrix}^T, & s\in[1/3,2/3], \\
			\begin{pmatrix} 3(s-2/3), \xi_{i_1}, \xi_{i,2} \end{pmatrix}^T, & s\in[2/3,1];
		\end{cases}
	\end{alignat}
	see also \cite[Example~2.2]{GyongyKrylov2006} for a similar rewriting of splitting-up methods.
	Here, $(\xi_i,w_i)_{i=1,\dots,q}$ defines a $q$-point integration rule for a standard normal random variable in two dimensions of order $5$.
	The weights of the cubature formula are $\lambda_i=\lambda_{q+i}=\frac{w_i}{2}$.
	This cubature formula is weakly symmetric, as $\sum_{i=1}^{q}w_i\xi_{i,1}=\sum_{i=1}^{q}w_i\xi_{i,2}=0$.
	The resulting ordinary differential equations can be solved exactly, see \cite{LordKoekkoekVanDijk2010,BayerFrizLoeffen2010} for the analytic formulas.

	For simplicity, we assume that $(\xi_i,w_i)_{i=1,\dots,q}$ is the tensor product Gauss-Hermite quadrature of order $5$.
	This implies $q=9$, and hence, we have $18$ paths per time interval.
	More efficient quadratures, in particular for the case of a high-dimensional driving Brownian motion, can be found in \cite{Stroud1971}.

	The numerical results are given in Figure~\ref{fig:heston_moments}. 
	We observe the second order rate of convergence.
	With $8$ time steps, we obtain a relative error of less than $10^{-2}$ for all quantities of interest.
	\begin{figure}[htpb]
		\begin{center}
			\includegraphics{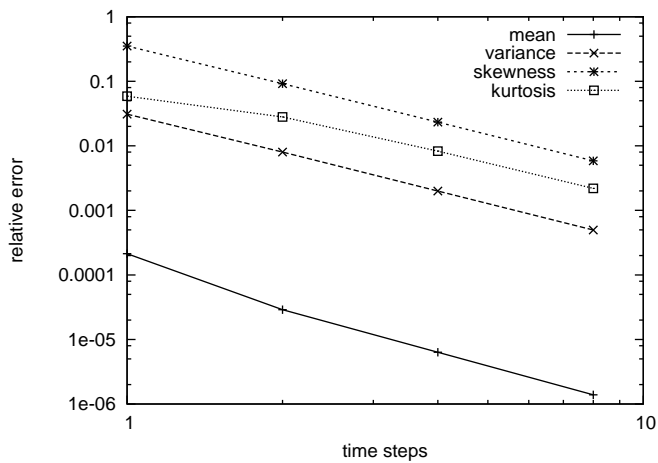}
		\end{center}
		\caption{Convergence of mean, variance, skewness and kurtosis of the log price in the Heston model}
		\label{fig:heston_moments}
	\end{figure}
\end{example}
\begin{example}
	The Heath-Jarrow-Morton framework is included in our setup.
	As explained in \cite{DoersekTeichmann2011}, it is more natural to use $\cosh$-weighted spaces instead of polynomially weighted spaces in this case.
	The more general definition of vector fields in Definition~\ref{def:vectorfield} allows us to enlarge the class of admissible equations considerably compared to \cite{DoersekTeichmann2011}.
\end{example}

\subsection{Smoothing effects under the UFG condition}
It is proved in \cite{Kusuoka2004,LyonsVictoir2004} that under the UFG condition, we obtain the optimal rate of convergence even for nonsmooth payoffs $f$ by using non-uniform time steps due to the smoothing effects of $P_t f$ in the direction of the vector fields $V_j$.
The aim of this section is to show that a corresponding result holds true for unbounded payoffs, as well.  
In particular, we will focus on exponentially growing payoffs through the choice of $\cosh(\alpha\lVert x\rVert)$ as the weight function.
This has important applications in mathematical finance, where one frequently models the log price as the solution of a stochastic differential equation, and thus, all payoffs will be a function of the exponential of the stochastic process.

We assume that we are in the finite dimensional situation, $H=\mathbb{R}^N$ for some $N\in\mathbb{N}$, and that $A=0$.
Suppose that all vector fields $V_j\colon\mathbb{R}^N\to\mathbb{R}^N$ are bounded and $\mathrm{C}^{\infty}$-bounded.
We choose the D-admissible weight function $\psi(x):=\cosh(\alpha\lVert x\rVert)$ for some $\alpha>0$.
\begin{proposition}
	\label{prop:coshalphacompat}
	For any $\alpha>0$, there exists $C>0$ such that
	\begin{equation}
		\mathbb{E}[\cosh(\alpha\lVert X^x_t\rVert)]
		\le
		\exp(Ct)\cosh(\alpha\lVert x\rVert).
	\end{equation}
\end{proposition}
\begin{proof}
	Note that for any $k\in\mathbb{N}$,
	\begin{equation}
		D^k \cosh(\alpha\lVert x\rVert)(h_1,\dots,h_k)
		\le
		C_k\cosh(\alpha\lVert x\rVert)\prod_{j=1}^{k}\lVert h_j\rVert.
	\end{equation}
	With $\mathcal{G}$ the generator of $P_t$, we thus obtain from the boundedness of the vector fields that $\mathcal{G}\cosh(\alpha\lVert x\rVert)\le C\cosh(\alpha\lVert x\rVert)$.
	It follows that
	\begin{alignat}{2}{}
		\mathbb{E}[\cosh(\alpha\lVert X^x_t\rVert)]
		&=
		\cosh(\alpha\lVert x\rVert) + \int_{0}^{t}\mathbb{E}[\mathcal{G}\cosh(\alpha\lVert X^x_s\rVert)]\dd s \\
		&\le
		\cosh(\alpha\lVert x\rVert) + \int_{0}^{t}C\mathbb{E}[\cosh(\alpha\lVert X^x_s\rVert)]\dd s.
	\end{alignat}
	The Gronwall inequality now proves the claim.
\end{proof}
\begin{corollary}
	\label{cor:coshalphapcompat}
	For any $p\in[1,\infty)$ and $T>0$, there exists $C_T>0$ such that
	\begin{equation}
		\mathbb{E}[\cosh(\alpha\lVert X^x_t\rVert)^p]^{1/p}
		\le
		C_T\cosh(\alpha\lVert x\rVert)
		\quad\text{for all $t\in[0,T]$}.
	\end{equation}
\end{corollary}
\begin{proof}
	We only need to note that for any $p\in[1,\infty)$, there exists some constant $C>0$ with $C^{-1}\cosh(pu)\le\cosh(u)^p\le C\cosh(pu)$ for all $u\in[0,\infty)$, and apply Proposition~\ref{prop:coshalphacompat}.
\end{proof}
We formulate now the ellipticity assumptions that are necessary to obtain smoothing effects.
We follow \cite{CrisanGhazali2007}.
\begin{description}
	\item[The UFG condition]
		There exists $\ell\in\mathbb{N}$ such that for every $\alpha\in\mathcal{A}^{*}$, there exist $\varphi_{\alpha,\beta}\in\mathrm{C}_b^{\infty}(\mathbb{R}^N)$, $\beta\in\mathcal{A}_\ell^{*}$, such that
		\begin{equation}
			V_{[\alpha]}
			=
			\sum_{\beta\in\mathcal{A}_\ell^{*}}\varphi_{\alpha,\beta}V_{[\beta]}.
		\end{equation}
	\item[The V0 condition]
		For some $\varphi_{\beta}\in\mathrm{C}_b^{\infty}(\mathbb{R}^N)$, $\beta\in\mathcal{A}_2^{*}$,
		\begin{equation}
			V_0
			=
			\sum_{\beta\in\mathcal{A}_2^{*}}\varphi_{\beta}V_{[\beta]}.
		\end{equation}
\end{description}
\begin{theorem}
	\label{thm:smoothingBpsi}
	Assume that the UFG and V0 conditions are satisfied.
	Then, for any $f\in\mathrm{C}_b^{\infty}(\mathbb{R}^N)$, any $k,m\ge 0$ and any $i_1,\dots,i_{k+m}=0,1,\dots,d$,
	\begin{equation}
		\lVert V_{i_1}\dots V_{i_k}P_t V_{i_{k+1}}\dots V_{i_{k+m}}f\rVert_{\psi}
		\le
		Ct^{-\deg(i_1,\dots,i_{k+m})/2}\lVert f\rVert_{\psi}.
	\end{equation}
\end{theorem}
\begin{proof}
	We can apply \cite[Corollary~2.17]{KusuokaStroock1987} to obtain that for each $x\in\mathbb{R}^N$, there exists a real-valued random variable $\pi^x$, depending on $k$ and $i_1,\dots,i_{k+m}$, such that
	\begin{equation}
		V_{i_1}\dots V_{i_k}P_t V_{i_{k+1}}\dots V_{i_{k+m}} f(x)
		=
		\mathbb{E}[f(X^x_t)\pi^{x}].
	\end{equation}
	Furthermore, for each $p\in[1,\infty)$, there exists a constant $C>0$ with
	\begin{equation}
		\sup_{x\in\mathbb{R}^N}\mathbb{E}[\lvert\pi^x\rvert^p]
		\le
		Ct^{-\deg(i_1,\dots,i_{k+m})/2}.
	\end{equation}
	It follows that for $p$, $q\in(1,\infty)$ with $\frac{1}{p}+\frac{1}{q}=1$,
	\begin{alignat}{2}{}
		\lVert V_{i_1}\dots V_{i_k}P_t V_{i_{k+1}}\dots V_{i_{k+m}}f\rVert_{\psi}
		&\le
		\sup_{x\in\mathbb{R}^N}\psi(x)^{-1}\mathbb{E}[\lvert f(X^x_t)\rvert\cdot\lvert\pi^x\rvert] \\
		&\le
		\lVert f\rVert_{\psi}\sup_{x\in\mathbb{R}^N}\psi(x)^{-1}\mathbb{E}[\psi(X^x_t)^{p}]^{1/p}\cdot\mathbb{E}[\lvert\pi^x\rvert^{q}]^{1/q} \notag \\
		&\le
		Ct^{-\deg(i_1,\dots,i_{k+m})/2}\lVert f\rVert_{\psi}, \notag
	\end{alignat}
	where we apply Corollary~\ref{cor:coshalphapcompat}.
\end{proof}
\begin{corollary}
	\label{cor:nonunifgridBpsi}
	Assume that the UFG and V0 conditions are satisfied.
	Then, for any mesh $0=t_0<\dots<t_n=T$ and $f\in\mathrm{C}_b^{\infty}(\mathbb{R}^N)$,
	\begin{alignat}{2}{}
		\lVert P_T f - &Q_{t_1-t_0}\dotsm Q_{t_n-t_{n-1}}f\rVert_{\psi} \\
		&\le
		C\lVert\nabla f\rVert_{\psi}\left( (t_n-t_{n-1})^{1/2}+\sum_{i=1}^{n-1}\frac{(t_i-t_{i-1})^{(m+1)/2}}{(T-t_i)^{m/2}} \right). \notag
	\end{alignat}
\end{corollary}
\begin{proof}
	We proceed as in the proof of \cite[Proposition~3.6]{LyonsVictoir2004}.
	Assume first that $f\in\mathrm{C}_b^{\infty}(\mathbb{R}^N)$; the general result then follows from a density argument.
	First, note that
	\begin{alignat}{2}{}
		\lVert P_{\Delta t} f-f\rVert_{\psi}
		&\le
		\sup_{x\in\mathbb{R}^N}\psi(x)^{-1}\mathbb{E}[\lvert f(X^x_{\Delta t})-f(x)\rvert] \\
		&\le
		\sup_{x\in\mathbb{R}^N}\psi(x)^{-1}\mathbb{E}[\sup_{s\in[0,1]}\lVert\nabla f(sX^x_{\Delta t}+(1-s)x)\rVert\cdot\lVert X^x_{\Delta t}-x\rVert] \notag \\
		&\le
		\lVert \nabla f\rVert_{\psi}\sup_{x\in\mathbb{R}^N}\psi(x)^{-1}\mathbb{E}[\sup_{s\in[0,1]}\psi(sX^x_{\Delta t}+(1-s)x)^2]^{1/2}\mathbb{E}[\lVert X^x_{\Delta t}-x\rVert^2]^{1/2}. \notag
	\end{alignat}
	As $\lVert sX^x_{\Delta t}+(1-s)x\rVert\le\max(\lVert X^x_{\Delta t}\rVert,\lVert x\rVert)$ for all $s\in[0,1]$ and $\cosh$ is monotonic on $[0,\infty)$, we see that Corollary~\ref{cor:coshalphapcompat} yields
	\begin{equation}
		\lVert P_{\Delta t} f-f\rVert_{\psi}
		\le C(\Delta t)^{1/2}\lVert \nabla f\rVert_{\psi}.
	\end{equation}
	By Theorem~\ref{thm:smoothingBpsi}, we obtain
	\begin{alignat}{2}{}
		\lVert (P_{\Delta t}-Q_{\Delta t})P_{T-t} f\rVert_{\psi}
		&\le
		(\Delta t)^{(m+1)/2}\sum_{\substack{(i_1,\dots,i_k)\in\mathcal{A}_m\\(i_0,i_1,\dots,i_k)\notin\mathcal{A}_m}}\lVert V_{i_0}V_{i_1}\dots V_{i_k}P_{T-t} f\rVert_{\psi} \notag \\
		&\le
		C(\Delta t)^{(m+1)/2}(T-t)^{-m/2}\lVert \nabla f\rVert_{\psi}.
	\end{alignat}
	Summing up in the usual manner, the claim follows.
\end{proof}
\begin{corollary}
	Under the UFG and V0 assumptions, the cubature method converges of optimal order for $f\in\mathcal{B}^{\psi}(\mathbb{R}^N)$ with $\nabla f\in\mathcal{B}^{\psi}(\mathbb{R}^N)$ on graded meshes such as the ones chosen in \cite[Example~3.7]{LyonsVictoir2004}.
\end{corollary}
\begin{proof}
	This follows directly from Corollary~\ref{cor:nonunifgridBpsi} together with the density of $\mathrm{C}_b^{\infty}(\mathbb{R}^N)$ in $\mathcal{B}^{\psi}(\mathbb{R}^N)$.
\end{proof}

\appendix
\section{A counterexample}
\label{sec:counterexample}
Not every admissible weight function is D-admissible, as already the counterexample $X=\mathbb{R}$, $\psi(x):=1+x^2+x^{-1}\chi_{(0,\infty)}$ with $\chi_A(x):=1$ for $x\in A$ and $0$ for $x\notin A$ the indicator of $A$ from \cite[Remark~4.6]{DoersekTeichmann2010} shows.
However, such an assumption is necessary to be able to transfer differentiability properties to limits when using weighted supremum norms.

Let us consider a concrete example.
Choose the admissible weight function $\psi(x):=1+x^2\chi_{(-\infty,0)}+x^{-2}\chi_{(0,\infty)}$ on $X=\mathbb{R}$.
Choose bounded, smooth functions $f_n$, $n\in\mathbb{N}$, by
\begin{equation}
	f_n(x):=
	\begin{cases}
		0, & x\le 0, \\
		1, & x\ge n^{-1},
	\end{cases}
\end{equation}
and monotone and smooth on $(0,n^{-1})$ such that $0\le f_n(x)\le 1$ for all $x\in\mathbb{R}$.
This can be done in such a way that $\lvert f_n'(x)\rvert\le Cn$ on $(0,n^{-1})$ for some $C>0$ independent of $n\in\mathbb{N}$, for example by choosing $f_1$ as required and setting $f_n(x):=f_1(nx)$.
Then, for $n\in\mathbb{N}$ and $m\ge n$,
\begin{alignat}{2}{}
	\lVert f_n - f_m \rVert_{\psi}
	&\le
	2\sup_{x\in(0,n^{-1})}\psi(x)^{-1}
	=
	2(1+n)^{-2} \quad\text{and} \\
	\lvert f_n \rvert_{\psi,1}
	&\le
	2\sup_{x\in(0,n^{-1})}\psi(x)^{-1}Cn
	=
	2C\frac{n}{(1+n)^2},
	\intertext{from which}
	\lvert f_n - f_m \rvert_{\psi,1}
	&\le
	2C\left( \frac{n}{(1+n)^2} + \frac{m}{(1+m)^2} \right).
\end{alignat}
It follows that $(f_n)_{n\in\mathbb{N}}$ is a Cauchy sequence in $\mathrm{B}^{\psi}_1(\mathbb{R})$.
As evaluation functionals are continuous, we see that the only candidate for the limit is $f=\chi_{(0,\infty)}$.
This function, however, is not in $\mathrm{B}^{\psi}_1(\mathbb{R})$, and is not even continuous.

This is not a contradiction to the characterization of $\mathcal{B}^{\psi}(\mathbb{R})$ by continuity, as no set $K_R:=\left\{ x\in\mathbb{R}\colon \psi(x)\le R \right\}$ contains a neighborhood of $x=0$.


\def\cprime{$'$}
\providecommand{\bysame}{\leavevmode\hbox to3em{\hrulefill}\thinspace}
\providecommand{\MR}{\relax\ifhmode\unskip\space\fi MR }
\providecommand{\MRhref}[2]{%
  \href{http://www.ams.org/mathscinet-getitem?mr=#1}{#2}
}
\providecommand{\href}[2]{#2}

\end{document}